\documentclass[english,10pt]{article}

\usepackage[english]{babel}
\usepackage[utf8x]{inputenc}
\usepackage[T1]{fontenc}
\usepackage[formats]{listings}
\usepackage{geometry}
\usepackage{enumitem,csquotes}

\usepackage{graphicx}
\usepackage[colorinlistoftodos]{todonotes}
\usepackage{setspace}
\usepackage{placeins}
\usepackage{enumitem}
\usepackage{titlesec}
\usepackage{cite}
\usepackage{comment}
\usepackage{caption}
\usepackage{subcaption}
\usepackage{enumitem}

\usepackage[colorlinks=true, allcolors=blue]{hyperref}

\usepackage{amsmath}
\usepackage{amsthm}
\usepackage{amssymb}
\usepackage{amsfonts}
\usepackage{bbm}
\usepackage{bm}
\usepackage{nicefrac}
\usepackage{mathtools}

\usepackage{tikz}

\usepackage{graphicx}
\usepackage{fullpage}
\usepackage{float}
\usepackage{wrapfig}
\usepackage{caption}

\usepackage{algpseudocode}
\usepackage{algorithm}
\usepackage{listings}

\usepackage{amsthm}
\usepackage{dsfont}
\usepackage{array}
\usepackage{mathrsfs}
\usepackage{cite}
\usepackage{comment}
\usepackage{mathrsfs}

\makeatother

\usepackage{babel}
\usepackage{color}
\usepackage{centernot}

\usepackage{babel}
\usepackage{sansmath}

\newcommand{\bbeta}{\bm{\beta}}
\newcommand{\btheta}{\bm{\theta}}
\newcommand{\dkl}{\mathrm{D_{KL} }}
\newcommand{\R}{\mathbb{R}}
\newcommand{\E}{\mathbb{E}}
\renewcommand{\P}{\mathbb{P}}

\title{Variational Inference in high-dimensional linear regression}

\author{Sumit Mukherjee\thanks{Department of Statistics, Columbia University, New York, NY 10027, U.S.A. \tiny{sm3949@columbia.edu}} \and 
  Subhabrata Sen\thanks{Department of Statistics, Harvard
    University, Cambridge, MA 02138, U.S.A. \tiny{subhabratasen@fas.harvard.edu}} }

\theoremstyle{plain}\newtheorem{lemma}{\textbf{Lemma}}\newtheorem{theorem}{\textbf{Theorem}}\newtheorem{corollary}{\textbf{Corollary}}\newtheorem{definition}{\textbf{Definition}}

\theoremstyle{definition}

\theoremstyle{definition}\newtheorem{remark}{\textbf{Remark}}

\makeatletter
\renewenvironment{proof}[1][\proofname] {
	\par\pushQED{\qed}\normalfont
	\topsep6\p@\@plus6\p@\relax
	\trivlist\item[\hskip\labelsep\bfseries#1\@addpunct{:}]
 	\ignorespaces
} {
	\popQED\endtrivlist\@endpefalse
}
\makeatother

\begin{document} 

\maketitle

\abstract{We study high-dimensional Bayesian linear regression with product priors. Using the nascent theory of \emph{non-linear large deviations} \cite{chatterjee2016nonlinear}, we derive sufficient conditions for the leading-order correctness of the naive mean-field approximation to the log-normalizing constant of the posterior distribution. Subsequently, assuming a true  linear model for the observed data, we derive a limiting infinite dimensional variational formula for the log normalizing constant of the posterior. Furthermore, we establish that under an additional ``separation" condition, the variational problem has a unique optimizer, and this optimizer governs the probabilistic properties of the posterior distribution. We provide intuitive sufficient conditions for the validity of this ``separation" condition. Finally, we illustrate our results on concrete examples with specific design matrices. 
}

\section{Introduction}
In this age of big-data, statisticians routinely analyze large, high-dimensional datasets arising from applications in genomics, finance, public policy etc., with the goal of discovering relationships between the response variable, and the observed features. The linear regression model 
\begin{align*}
\mathbf{y} = \mathbf{X} \bbeta + \varepsilon
\end{align*}
is arguably the most common framework for this task when the response is continuous. Under a Bayesian formalism, the statistician posits a prior distribution $\pi$ for the coefficient vector $\bbeta$, and constructs the corresponding posterior. Subsequent inference is based solely on this posterior distribution. Two questions are naturally relevant in this setting:
\begin{enumerate}
\item What are the statistical properties of Bayesian procedures, particularly in high-dimensions? 
\item Are these procedures computationally tractable for large datasets with high-dimensional features? 
\end{enumerate}

The theoretical performance of high-dimensional Bayesian methods have been examined extensively in recent times.  
In Bayesian asymptotic theory, given data $(y_i, \mathbf{x}_i)_{i=1}^{n}$, $x_i \in \mathbb{R}^{p}$, one assumes the correctness of a frequentist model $y_i = \mathbf{x}_i^{\mathrm{T}}\bbeta_0+ \varepsilon$, and studies frequentist inference of the coefficient vector $\beta_0$ under Bayesian procedures. \cite{ghosal1999asymptotic} established a version of the traditional Bernstein-Von-Mises theorem as long as $p^4\log p/n \to 0$. A lot of  recent attention has been focused on the high-dimensional regime $p\gg n$ with an additional assumption on the sparsity of $\bbeta_0$. In this context, spike-and-slab based approaches  \cite{mitchell1988bayesian,ishwaran2005spike} have been established to exhibit optimal frequentist properties \cite{castillo2015bayesian}. In particular, these posterior distributions ``contract" to the underlying truth $\bbeta_0$ at the minimax optimal estimation rate (we refer the interested reader to \cite{banerjee2021bayesian} for a formal definition of posterior contraction, and a survey of recent breakthroughs in this area). Despite the superior theoretical properties of this approach, the sharp spike-and-slab based approaches suffer from a computational bottleneck. In the special case of linear regression, continuous shrinkage priors (see e.g. \cite{carvalho2010horseshoe}) provide a comparatively more tractable alternative from the computational perspective\cite{bhattacharya2016fast}. Contraction properties of the corresponding posteriors were characterized in \cite{song2017nearly}. Unfortunately, this strategy is specific to the linear model, and does not generalize beyond this setting.

In general, computationally tractable Bayesian inference for high-dimensional models presents significant challenges. MCMC based strategies have been explored extensively for this purpose. Despite rapid progress in MCMC methodology and supporting theory, these methods are still slower than competing frequentist methodology e.g. those based on convex optimization. Variational methods \cite{wainwright2008graphical} provide an attractive general option to the Bayesian statistician. The simplest form of variational inference approximates the true posterior distribution using a product distribution---this version is often referred to as \emph{naive mean-field Variational Bayes} (nVB). Computing the best approximating product distribution is computationally fast, and thus these methods provide a practical option for modern datasets. We refer the interested reader to \cite{blei2017variational} for an introduction to variational Bayes methods in statistics and machine learning.

Although variational methods provide a computationally feasible strategy, supporting theoretical evidence has been relatively scarce.  In  \cite{wang2006convergence,wang2019frequentist}, the authors established the correctness of this approach in parametric models in the classical fixed $p$ setting. Subsequently \cite{wang2019variational} study variational Bayes in misspecified models. In the context of the linear model, early work by \cite{neville2014mean,ormerod2017variational} focused on variational inference in the low dimensional linear model,  while \cite{carbonetto2012scalable} provides an early variational approximation for variable selection.

The past two years have witnessed rapid progress in the analysis of variational methods for high-dimensional models \cite{alquier2016properties,alquier2020concentration,han2019statistical,ray2021variational,ray2020spike,yang2020alpha}. These results focus on high-dimensional models with a sparse underlying truth $\bbeta_0$, and study the contraction properties of the variational posterior. In particular, they derive sufficient conditions for the variational posterior to contract at the minimax optimal rate. Correctness of variational methods have also been established for community detection \cite{bickel2013asymptotic,zhang2020theoretical}, a poisson mixed model  \cite{hall2011asymptotic}, frequentist models \cite{westling2019beyond} and mixture models \cite{cherief2018consistency}.  

In sharp contrast, nVB methods can fail in truly high-dimensional settings, e.g. in versions of topic modeling \cite{ghorbani2019instability}. In this specific situation, the correct variational approximation is provided by the TAP approximation from spin-glass theory \cite{mezard1987spin}. \cite{fan2018tap} establishes rigorous guarantees regarding the validity of the TAP approach in the context of the simpler $\mathbb{Z}_2$ synchronization problem. 

There is thus an immediate need to understand general properties of statistical problems which ensure the correctness of the nVB approximation. In this paper, we study the accuracy of this approximation in Bayesian linear regression with product priors. We provide easily verifiable conditions which ensure the asymptotic accuracy of the nVB approximation. Further, we illustrate that the approximation can yield detailed information regarding the statistical properties of this model, which might be unavailable using other techniques. We elaborate on our specific contributions below.

\subsection{Contributions} 
Our main contributions in this paper are as follows: 
\begin{itemize}
\item[(i)] In Theorem \ref{thm:main}, we provide sufficient conditions for the asymptotic tightness of the nVB approximation. The conditions are easily verifiable, and can be explicitly checked in specific applications. This provides rigorous theoretical support for widely used mean-field approximation based methodology.  

\item[(ii)] Assuming a true frequentist linear model for the data, we derive a limiting variational formula for the log normalizing constant in Theorem \ref{thm:var_conv_new}. We emphasize that in contrast to existing results, we do not assume  any sparsity on the true regression coefficients. We refer the interested reader to  \cite{lee2020continuous} for a motivating discussion on the importance of such scenarios in scientific applications. 

\item[(iii)] Under an additional ``separation" condition \eqref{eq:separate}, we establish that the limiting variational problem has a unique optimizer (see Theorem \ref{thm:pure_state} for a precise statement).  Further, this optimizer governs the probabilistic properties of the posterior distribution (Corollary \ref{cor:eg}).  We also provide interpretable sufficient conditions which enforce this separation condition (Lemma \ref{lemma:pure_state}). 

We emphasize that in existing analyses of high-dimensional Bayesian linear regression, properties of the posterior are directly established \cite{castillo2015bayesian}, and the variational posterior is analyzed independently \cite{ray2021variational}. Our approach is inherently different---we first establish the correctness of the mean-field approximation, and then study the posterior through the lens of the mean-field approximation formula. 

\item[(iv)] We further illustrate our general results by applying them to three specific examples---a two factor ANOVA model, a gaussian design setting with spiked covariance, and a sparse bernoulli design. In each case, we identify the specific limiting functional which determines the limiting log normalizing constant.

 
\item[(v)] From a theoretical perspective, our results crucially utilize recent advances in the theory of \emph{non-linear large deviations}. Initiated in the seminal paper \cite{chatterjee2016nonlinear}, the theory of non-linear large deviations was originally conceived to answer some deep questions concerning large deviations of sub-graph counts in sparse random graphs. In \cite{basakmukherjee}, one of the authors successfully utilized this framework to establish the tightness of the naive mean field approximation for the log-partition function in a family of Potts models. To the best of our knowledge, these developments have not been utilized previously for other statistical models. In this paper, we demonstrate the usefulness of these tools in the context of high dimensional statistics; we hope that this spurs an in-depth study of their applicability for other high dimensional problems. We consider this to be a key conceptual contribution of this paper. 
\end{itemize}



\subsection{Non-linear large deviations and related results} 
In a breakthrough paper, Chatterjee and Dembo \cite{chatterjee2016nonlinear} introduced the theory of \emph{non-linear large deviations} with the goal of studying large deviations for non-linear functions of bernoulli random variables. As an application of this general machinery, they characterized sharp deviation probabilities for sub-graph counts in sparse Erd\H{o}s-R\'{e}nyi random graphs. Subsequent extensions by Eldan \cite{eldan2018gaussian} and Augeri \cite{augeri2019transportation} allow one to track a wider regime of sparsity for the binary variables. In a different direction, Yan \cite{yan} extended the Chatterjee-Dembo framework to general bounded Banach-space valued variables. See also \cite{austin2019structure} for related decompositions of general Gibbs measures using information theoretic ideas. These results have galvanized the study of large deviations for sub-graph counts on sparse random graphs, and the past three years have witnessed rapid progress in this direction. We refer the interested reader to \cite{cook2021regularity} and references therein for a survey of recent progress in this area. 

At the heart of the Chatterjee-Dembo framework lies a tight approximation bound for the log-normalizing constant for general Gibbs type distributions in terms of the naive mean-field approximation formula. This framework was utilized by one of the authors \cite{basakmukherjee} to derive asymptotic limits for the log normalizing constant of Potts models on several sequences of graphs. Similar results were independently derived by \cite{jain2018mean,jain2019mean} using different techniques. 

In this paper, we study Bayesian linear regression through the lens of non-linear large deviations, and uncover precise statistical properties of these models by analyzing the mean-field variational problem.

\subsection{Setup} We observe $\{ (y_i, x_i): 1 \leq i \leq n\}$, $y_i \in \mathbb{R}$, $x_i \in \mathbb{R}^p$. We set $\mathbf{y} = (y_i) \in \mathbb{R}^n$ and $\mathbf{X}^{\mathrm{T}} = [x_1, \cdots, x_n]$.  Throughout, we work in an asymptotic setting where $p$ and $n=n(p) $ are both going to $\infty$. A natural Bayesian model for such a dataset assumes 
\begin{align}
&\beta_1, \cdots , \beta_p \sim^{iid} \pi, \,\,\,\, \bbeta^{\mathrm{T}} = (\beta_1, \cdots, \beta_p),\,\,\,\,\ \mathbf{y} = \mathbf{X} \bbeta + \mathbf{\varepsilon}, \,\,\, \varepsilon \sim \mathcal{N}(0, \sigma^2 I_n). \label{eq:model} 
\end{align}
In the display above, $\pi$ is a  probability distribution supported on $[-1,1]$, and we consider an iid prior on the regression coefficients. For our subsequent discussion, the precise interval $[-1,1]$ for the prior support is not crucial--- our arguments  go through unchanged, as long as the prior has bounded support. Note that in the context of Bayesian inference for linear regression, the regression parameters are often drawn from a parametric family, and the parameters specifying the prior are, in turn, sampled from a hyper-prior. In our discussions, we will restrict ourselves to the simpler setting where the prior $\pi$ is fixed---however, we do not make any assumptions on $\pi$ apart from the bounded support assumption. Throughout, we assume that the noise variance $\sigma^2>0$ is fixed and known to the statistician. 

Given the Bayesian model \eqref{eq:model}, one naturally constructs the posterior distribution 
\begin{align}
\frac{\mathrm{d} \mu_{\mathbf{y}, \mathbf{X}}}{\mathrm{d}\pi^{\otimes p}}(\bbeta) \propto \exp\Big( -\frac{1}{2\sigma^2} \| \mathbf{y} - \mathbf{X} \bbeta \|_2^2 \Big) \propto \exp \Big( - \frac{1}{2\sigma^2} \Big[ \bbeta^{\mathrm{T}} A_p\bbeta +\bbeta^{\mathrm{T}} D_p \bbeta -2 \mathbf{z}^{\mathrm{T}} \bbeta\Big] \Big),  \nonumber 
\end{align}
where $\mathbf{z} = \mathbf{X}^{\mathrm{T}} \mathbf{y}$, and $D_p,A_p$ are the diagonal and off-diagonal matrices obtained from $ \mathbf{X}^{\mathrm{T}} \mathbf{X} $. More precisely,
\begin{eqnarray*}
&A_p(i,i):=0, \quad &D_p(i,i):=\Big(\mathbf{X}^{\mathrm{T}} \mathbf{X}\Big)_{ii},\\
&A_p(i,j):=\Big(\mathbf{X}^{\mathrm{T}} \mathbf{X}\Big)_{ij},\quad &D_p(i,j):=0.
\end{eqnarray*} 
Since the term $\bbeta^{\mathrm{T}} D\bbeta$ is additive in the components of $\bbeta$, we can absorb the term $\bbeta^{\mathrm{T}}D\bbeta$ in the base measure $\pi^{\otimes p}$, via the following definition:  
\begin{definition}[Exponential Family]
\label{def:exp}
For any $\gamma:=(\gamma_1,\gamma_2)\in \mathbb{R}^2$ define a probability measure $\pi_{\gamma}$ on $[-1,1]$ as 
\[\frac{\mathrm{d}\pi_\gamma}{\mathrm{d}\pi}(z):= \exp\left(\gamma_1 z-\frac{\gamma_2}{2}z^2-c(\gamma)\right),\quad c(\gamma):=\log\int_{[-1,1]} \exp\left(\gamma_1 z - \frac{\gamma_2}{2}z^2\right) \mathrm{d}\pi(z).\]
\end{definition}
Using this definition we can write the posterior distribution $\mu$ as 
\begin{align*}
 \mu_{\mathbf{y}, \mathbf{X}} (\mathrm{d} \bbeta) \propto  \exp \Big( - \frac{1}{2\sigma^2} \Big[ \bbeta^{\mathrm{T}} A \bbeta -2 \mathbf{z}^{\mathrm{T}} \bbeta\Big] \Big) \prod_{i=1}^{p} \pi_i( \mathrm{d}\beta_i),\quad \text{where } \pi_i:=\pi_{(0,d_i)}, \,\, d_i := \frac{D_{ii}}{\sigma^2}. 
\end{align*}
A central object in the theoretical study of these posterior distributions is the normalizing constant, also referred to as the ``partition function" in statistical physics parlance. Formally, we define 
\begin{align}
Z_{p}(\mathbf{y}, \mathbf{X}) = \int_{[-1,1]^{p}} \exp\Big( - \frac{1}{2\sigma^2} \Big[ \bbeta^{\mathrm{T}} A \bbeta - 2 \mathbf{z}^{\mathrm{T}} \bbeta\Big] \Big) \prod_{i=1}^p\pi_i( \mathrm{d}\beta_i). \label{eq:normalizing}
\end{align}
The partition function is intractable for most priors, unless special conjugacy properties are satisfied between the prior and the likelihood. Henceforth, we suppress the dependence of $\mu$, $Z_p$  on $\mathbf{y}$, $\mathbf{X}$ whenever it is clear from the context.  The classical Gibbs variational principle characterizes the partition function as 
\begin{align}
\log Z_p = \sup_{Q} \Big(  \mathbb{E}_{Q}\Big[ - \frac{1}{2\sigma^2} \Big[ \bbeta^{\mathrm{T}} A \bbeta - 2 \mathbf{z}^{\mathrm{T}} \bbeta\Big] \Big] - \dkl(Q \| \prod_{i=1}^p\pi_i) \Big), \nonumber 
\end{align}
where the supremum ranges over probability distributions $Q$ on $[-1,1]^{p}$ (see e.g. \cite{wainwright2008graphical}). 
In fact, the supremum in the above variational problem is attained by $Q= \mu$. 
The naive mean-field approximation to $Z_p$ restricts the supremum to product distributions, and thus obtains a universal lower bound. Formally, we have, 
\begin{align}
\log Z_p \geq \sup_{Q = \prod_{i=1}^{p} Q_i}  \Big(  \mathbb{E}_{Q}\Big[ - \frac{1}{2\sigma^2} \Big[ \bbeta^{\mathrm{T}} A \bbeta - 2 \mathbf{z}^{\mathrm{T}} \bbeta\Big] \Big] - \dkl(Q \| \prod_{i=1}^p\pi_i) \Big). \nonumber 
\end{align}
The optimizer in the display above provides the best approximation to $\mu$ among product distributions under KL divergence. This paper focuses on the tightness of the naive mean-field lower bound in the context of linear regression. 
For an in-depth survey of variational inference, we refer the interested reader to \cite{bishop2006pattern,blei2017variational,wainwright2008graphical}.

%
%
%

 
 \noindent
 \textbf{Outline:} The rest of the paper is structured as follows. We collect our results in Section \ref{sec:results}. We discuss some directions for future enquiry in Section \ref{sec:discussions}. The main results are established in Sections \ref{sec:proofs}. We defer some proofs to the Appendix. 
 
 \noindent
 \textbf{Acknowledgments:} The authors thank Pragya Sur for discussions on high-dimensional regression. SM gratefully thanks NSF (DMS 1712037) for support during this research. 

\section{Results} 
\label{sec:results} 
We collect our main results in this section. To this end, we first discuss some elementary facts regarding exponential families. 
\noindent 
The next result collects some analytic properties of the cumulant generating function $c(\cdot)$, which will be relevant for our subsequent discussion. For the sake of completeness, we provide a proof in Section \ref{sec:proofs}. 
\begin{lemma}
\label{lemma:properties} 
Let $c(\cdot):\mathbb{R}^2 \mapsto (-1,1)$ be as in Definition \ref{def:exp}. Assume that both $\{-1,1\}$ belong to the support of $\pi$. Then the following conclusions hold. 

\begin{enumerate}
\item[(i)] $\dot{c}(\gamma_1,\gamma_2):=\frac{\partial c(\gamma_1,\gamma_2)}{\partial \gamma_1}$ is strictly increasing in $\gamma_1$, with $\lim_{\gamma_1\rightarrow\pm \infty}\dot{c}(\gamma_1,\gamma_2)=\pm 1$ for every $\gamma_2\in \mathbb{R}$.

\item[(ii)]
 For any $t\in (-1,1)$, there exists a unique $h(t,\gamma_2)$ such that $\dot{c}(h(t,\gamma_2), \gamma_2)=t$. Further, $\lim_{t\rightarrow\pm1}h(t,\gamma_2)=\pm \infty$ for every $\gamma_2\in \mathbb{R}$.
 \end{enumerate}
\end{lemma}
\noindent 

\noindent
Armed with these basic facts, we can formally state the naive-mean field approximation to the log-normalizing constant \eqref{eq:normalizing}. 
\begin{definition}\label{def:G}
Define a possibly extended real valued function $G$ on $[-1,1]\times \mathbb{R}$ by setting
\begin{align*}
G(u,d):=&u h(u,d) -c(h(u,d),d) + c(0, d)&\text{ if }u\in (-1,1), d\in \mathbb{R},\\
:=&\dkl(\pi_\infty\| \pi_{(0,d)} )&\text{ if }u=1, d\in \mathbb{R},\\
:=&\dkl(\pi_{-\infty}\| \pi_{(0,d)} )&\text{ if }u=-1, d\in \mathbb{R},
\end{align*}
where $\pi_{\infty}$ and $\pi_{-\infty}$ are degenerate distributions which puts mass 1 at $1$ and $-1$ respectively.
\end{definition}

\noindent
We will need the following facts about the derivatives of $G$. We defer the proof of Lemma \ref{lemma:G_stability} to the Appendix. 

\begin{lemma}
\label{lemma:G_stability}
We have, for $u \in (-1,1)$ and $d \in \mathbb{R}$, 
\begin{align}
\frac{\partial G}{\partial u}(u,d) = h(u,d),&\,\,\, \frac{\partial G}{\partial d} = \frac{1}{2} \int_{-1}^{1} z^2 \mathrm{d} \pi_{(h(u,d),d)}(z) - \frac{1}{2} \int_{-1}^{1} z^2 \mathrm{d} \pi_{(0,d)}(z)  . \nonumber \\
\frac{\partial^2 G}{\partial^2 u}(u,d) &= \frac{1}{\ddot{c}(h(u,d),d)} >0.  \nonumber 
\end{align} 
Consequently, we have, $\sup_{u \in (-1,1), d \in \mathbb{R}} |\frac{\partial G}{\partial d} (u,d)| \leq \frac{1}{2}$. 
\end{lemma}

\begin{definition}\label{def:Mp} 
Define $M_p: [-1,1]^p \to \mathbb{R}$ as 
 \begin{align}
 M_p({\bf u}) :=\Big\{ - \frac{1}{2\sigma^2} \Big[ {\bf u}^{\mathrm{T}} A {\bf u}- 2 \mathbf{z}^{\mathrm{T}} {\bf u}\Big]-\sum_{i=1}^pG(u_i,d_i)\Big\}. \label{eq:Mp_defn} 
 \end{align} 
\end{definition} 

\begin{lemma} 
\label{lemma:naive_approx}
With $M_p(\cdot)$ as in \eqref{eq:Mp_defn}, we have, 
\begin{align} 
\sup_{Q = \prod_{i=1}^{p} Q_i}  \Big(  \mathbb{E}_{Q}\Big[ - \frac{1}{2\sigma^2} \Big[ \bbeta^{\mathrm{T}} A \bbeta - 2 \mathbf{z}^{\mathrm{T}} \bbeta\Big] \Big] - \dkl(Q \| \prod_{i=1}^p\pi_i) \Big) = \sup_{{\bf u}\in [-1,1]^p}M_p({\bf u}) \label{eq:variational_prob}.
\end{align}
\end{lemma}
\noindent 
This result follows immediately from elementary facts about exponential families. We refer the interested reader to \cite[Section 5.3]{wainwright2008graphical}.

\subsection{Validity of the naive mean-field approximation} 
 Throughout the paper, we use the usual Landau notation $O(\cdot)$ and $o(\cdot)$ for deterministic sequences dependent on $p$.
\begin{theorem}\label{thm:main}
Assume that the matrix $A_p$ satisfies the two conditions 
\begin{align}
\mathrm{tr}(A_p^2) &=o(p), \label{eq:mean_field}  \\
\sup_{{\bf u} \in [-1,1]^p}\sum_{i=1}^p\Big|\sum_{j=1}^p A_p(i,j) u_j \Big| &=O(p).\label{eq:weak}
\end{align}
\begin{enumerate}
\item[(i)] \label{thm:main_parta}
We have, setting $R_p:=\sup_{\mathbf{u} \in [-1, 1]^{p}} M_p ({\bf{u}})$, as $p \to \infty$,
\begin{align}\label{eq:mf_conclusion}
\log Z_p- R_p=o(p).\quad 
\end{align}

\item[(ii)]\label{thm:main_partb}
If $b_i:=\E_{\mu}(\beta_i|\beta_k,k\ne i)$, then the vector ${\bf b}:=(b_1,\ldots,b_p)$ satisfies
\begin{align*}
M_p({\bf b})-R_p= o(p).
\end{align*}

\item[(iii)] \label{thm:main_partc}
Suppose there exists $\hat{\bf u}\in [-1,1]^p$ which satisfies that for every $\eta>0$ we have
\begin{align}\label{eq:unique_mf}
\limsup_{p\to\infty}\frac{1}{p}\Big[\sup_{{\bf u}\in [-1,1]^p:\lVert {\bf u}- \hat{{\bf u}}\rVert_2^2\ge p\eta}M_p({\bf u})-R_p\Big]<0.
\end{align}
Assume further that the empirical measure $\frac{1}{p} \sum_i \delta_{D_p(i,i)}$ is uniformly integrable. 
Then, for any $\varepsilon>0$ and for any continuous function $\zeta:[-1,1]\times [0,1]\to\R$ we have
$$\mu_{\mathbf{y}, \mathbf{X}}\Big( \Big| \frac{1}{p}\sum_{i=1}^p\zeta\Big(\beta_i, \frac{i}{p}\Big)-\frac{1}{p}\sum_{i=1}^p \int_{[-1,1]}  \zeta\Big(z, \frac{i}{p}\Big) \pi_{(h(\hat{u}_i,d_i),d_i)}(dz) \Big| > \varepsilon  \Big) \to  0.$$
\end{enumerate}
\end{theorem} 
\begin{remark}
The careful reader would have already noticed that Theorem \ref{thm:main} is stated for \emph{deterministic} data $(\mathbf{y}, \mathbf{X})$. In practical applications, it is often more natural to assume that the data $(\mathbf{y}, \mathbf{X})$ is, in turn, sampled from some underlying distribution $\mathcal{P}$. 
The conclusions of Theorem \ref{thm:main} continue to hold as long as the sufficient conditions hold asymptotically with high probability under $\mathcal{P}$. 
\end{remark} 

\noindent
The condition \eqref{eq:weak} is extremely mild, and specifies that the log normalizing constant is asymptotically of order $p$. This condition is satisfied, for example, whenever the matrix $A_p = \mathbf{X}^{\mathrm{T}} \mathbf{X} - D_p$ has spectral norm $O(1)$. The assumption \eqref{eq:mean_field} is the non-trivial assumption in the statement above, and effectively guarantees the accuracy of the naive mean field approximation to leading exponential order. Note that if $ \lambda_1 \geq \cdots \geq \lambda_p$ denote the eigenvalues of $A_p$, then $\mathrm{tr}(A_p^2) = \sum_{i=1}^{p} \lambda_i^2$. Thus the requirement $\mathrm{tr}(A_p^2) =o(p)$ can be qualitatively interpreted to mean that the eigenstructure of $A$ is ``dominated" by a few top eigenvalues. Similar results were derived in the context of naive mean-field approximation for Potts models by one of the authors in \cite{basakmukherjee}.  Covariance matrices with an approximately low rank are ubiquitous in modern datasets, and we believe these conditions are satisfied in diverse applications of practical interest. Our result provides formal evidence to the correctness of widely used mean-field approximations in these settings. 
We prove Theorem \ref{thm:main} in Section \ref{subsec:proof_thm1}. 

The third part of our theorem provides further insights into the posterior distribution, assuming that the naive mean field approximation is ``dominated" by a unique factorized distribution. In the language of Statistical Physics, these distributions are referred to be in a ``pure phase". 


We now demonstrate a few applications of Theorem \ref{thm:main} to concrete examples, which cover both deterministic and random design matrices. We defer the proofs of these corollaries to Appendix \ref{sec:example_proofs}. 

\begin{corollary}\label{cor:det1}

Let $\mathbf{X}$ be any sequence of deterministic design matrices with ${\mathbf X}^{\mathrm T} {\mathbf X}=A_p+D_p$, where $A_p$ and $D_p$ represent the off diagonal and diagonal parts of ${\mathbf X}^{\mathrm T} {\mathbf X}$ as before. Suppose $A_p$ satisfies \eqref{eq:mean_field} and \eqref{eq:weak}, and the empirical measure $\frac{1}{p}\sum_{i=1}^pD_p(i,i)$ is uniformly integrable. Then the conclusion of Theorem \ref{thm:main} holds.
\end{corollary}

\begin{corollary}\label{cor:reg1}

Suppose that the $i^{th}$ row of the design matrix ${\mathbf X}$ equals $n^{-1/2}  {\mathbf x}_i$, where $\{\mathbf{x}_i\}_{1\le i\le n}\stackrel{i.i.d.}{\sim} N(0,\Gamma_p)$. Assume that $p=o(n)$, and the following conditions hold:

\begin{enumerate}
\item[(a)]
The off diagonal part $\Gamma_{p,{\rm off}}$ of the covariance matrix $\Gamma_p$ satisfies ${\rm tr}(\Gamma_{p,{\rm off}}^2)=o(p)$.

\item[(b)]
$\lVert \Gamma_p\rVert_2=O(1)$.
\end{enumerate}
\noindent
 Then the conclusion of Theorem \ref{thm:main} holds.
\end{corollary}

%
%
%
%

Sparse design matrices arise routinely in coding theory \cite{gallager1962low,mackay1999good} and genomics \cite{wu2011rare,tang2014large}. Our next corollary discusses the accuracy of the nVB approximation in the context of a linear regression problem with a sparse bernoulli design. We assume that the entries of the design are independent, but not necessarily identical. 

\begin{corollary}\label{cor:erdos1}

Suppose that the $(i,j)^{th}$ entry of the design matrix ${\mathbf X}$ equals $\sqrt{\frac{p}{n}} B(i,j)$, where $\{B(i,j)\}_{1\le i\le n,1\le j\le p}$ are mutually independent Bernoullis with $\P(B(i,j)=1)\le \frac{\lambda}{p}$, for some $\lambda>0$ free of $p$. If $p=o(n)$, the conclusion of Theorem \ref{thm:main} applies.

\end{corollary}

\subsection{Scaling limit for the log-normalizing constant}
Under the asymptotic validity of the naive mean-field approximation, we derive an asymptotic scaling limit for the log-normalizing constant. We will subsequently establish that this limiting description captures crucial information regarding the behavior of the optimizers at finite $n,p$. To this end, we will require some notation. 

The theory of dense graph limits was developed in  \cite{borgs2008convergent,borgs2012convergent,lovasz2007szemeredi}, and has received tremendous attention over the last decade in Probability, Combinatorics, Computer Science and Statistics. We refer the interested reader to \cite{lovasz2012large} for an in-depth survey of this area. 
Follow up work of Borgs et. al.~\cite{borgs2019L,borgs2018p} has extended  this theory significantly beyond the regime of dense graphs---the resulting $L^p$-convergence theory can handle sparse graphs and weighted matrices. Utilizing this set up, our next result will show that under the assumption that the off diagonal matrix $A_p$ obtained from $\mathbf{X}^{\mathrm{T}}\mathbf{X}$ converges in cut norm to a suitable graphon (need not be bounded), the corresponding log normalizing constant converges in probability to a deterministic optimization problem. We note that the use of cut-norms on matrices precedes the development of graph limit theory (see e.g. \cite{frieze1999quick} and references therein). 

\begin{definition}\label{def:cut}
A function $W:[0,1]^2 \mapsto \mathbb{R}$ is called symmetric if $W(x,y) = W(y,x)$ for all $x,y \in [0,1]$. Any symmetric function $W:[0,1]^2\mapsto \mathbb{R}$ which is $L^1$ integrable, i.e.~$\lVert W\rVert_1:=\int_{[0,1]^2}|W(x,y)|dxdy<\infty$ is called a graphon. Let $\mathcal{W}$ denote the space of all graphons.
\end{definition}			

\noindent 
The cut norm of a graphon $W$ is given by
			 $$\|{W}\|_\square=\Big|\sup_{S,T\subset [0,1]}\int_{S\times T}W(x,y)dxdy\Big|.$$
			  The cut norm is equivalent to the  $L^\infty\mapsto L^1$ operator norm defined by
			 \begin{align}\label{eq:cut}\lVert W\rVert_1:=\sup_{f,g: \lVert f\rVert_\infty, \lVert g\rVert_\infty \le 1}\Big|\int_{[0,1]^2} W(x,y)f(x)g(x)dxdy\Big|.\end{align}
More precisely, we have
		$\lVert{W}\rVert_\square\le\lVert{W}\rVert_{\infty\mapsto 1}\le 4\lVert{W}\rVert_\square.$
		It also follows from \eqref{eq:cut} that the cut norm is weaker than the $L^1$ norm, i.e. convergence in $L^1$ implies convergence in cut norm.

\begin{definition}
Given a symmetric $p\times p$ matrix $B$ with real entries, define a piecewise constant function on $[0,1]^2$ by dividing $[0,1]^2$ into $p^2$ smaller squares each of length $1/p$, and set
			\begin{align*}
			W_{B}(x,y):=&B(i,j)\text{ if }\lceil px\rceil =i, \lceil py\rceil =j \,\textrm{with}\,{ }i\ne j,\\
			=&0\textrm{ otherwise.} 
			\end{align*}
\end{definition}
\noindent
We will also need the following notion for embedding vectors into functions. 
\begin{definition}{\bf(Vector to function)}
Given  ${\bf t}:=(t_1,\cdots,t_p)\in \mathbb{R}^p$, define the piecewise constant function $w_{{\bf {t}} ,p}$ on $[0,1]$ by dividing $[0,1]$ into $p$ intervals $\cup_{i=1}^p\frac{1}{p}(i-1,i]$ of equal length $1/p$, and setting $w_{{\bf t}}(x):=t_i$ if $x$ is in the $i^{th}$ interval, i.e. $\lceil px\rceil =i$. 
\end{definition}

To derive the scaling limit, we will assume an underlying model $\mathbf{y} = \mathbf{X} {\bm \beta}_0 + \varepsilon$, where $\varepsilon \sim \mathcal{N}(0, \sigma^2 I)$. We say that a random variable $f:= f(\mathbf{y}, \mathbf{X}) \stackrel{P| \mathbf{X}}{\longrightarrow} 0$ if for any $\varepsilon>0$,  $\mathbb{P}\Big[ | f | > \varepsilon | \mathbf{X} \Big] \to 0$ as $p \to \infty$. Thus this convergence is conditional on the sequence of design matrices. 
%
%
In our subsequent analysis, the following representation lemma will be crucial. 
\begin{lemma}\label{lem:off_diagonal}

Suppose we are in the setting of Theorem \ref{thm:main}. Assume further that 
 \[ \quad \sum_{i=1}^pD_p(i,i)=O(p).\]
 Then there exists $(\xi_1,\cdots\xi_p)\stackrel{i.i.d.}{\sim}N(0,1)$ such that 
 \[ \frac{1}{p}\sup_{{\bf u}\in [-1,1]^p}\Big|M_p({\bf u})-\widetilde{M}_p({\bf u})\Big| \stackrel{P|\mathbf{X}}{\to} 0,\]
 where
 $$\widetilde{M}_p({\bf u}):=-\frac{1}{2\sigma^2}\Big[{\bf u}'A{\bf u}-2\sum_{i=1}^p u_i\sqrt{D_p(i,i)}\xi_i-2\bbeta_0' {\mathbf X}'{\mathbf X}{\bf u}\Big]-\sum_{i=1}^p G(u_i,d_i).$$

\end{lemma}
\noindent

We will derive a limiting formula for the log-normalizing constant \eqref{eq:normalizing} in terms of a variational problem on a space of probability distributions. This requires the following definition.

\begin{definition}
\label{def:limit_prob} 
Let $\mathscr{F}$ denote the space of all bounded measurable functions from $[0,1]\times \R$ to $[-1,1]$. 
Fixing $W\in \mathcal{W}, g,\psi\in L^1[0,1]$, define a functional $\mathcal{G}_{W,g,\psi}(\cdot) : \mathscr{F} \to \mathbb{R}$ by setting
\begin{align*}
\mathcal{G}_{W,g,\psi}(F):=&\frac{1}{\sigma^2} \left( - \frac{1}{2} \mathbb{E}[W(X, X') F(X,Z') F(X',Z')] + \mathbb{E}[g(X) F(X,Z) ]+ \mathbb{E}[\sqrt{\psi(X)} F(X,Z) Z] \right)\\
 -& \mathbb{E}\Big[G \Big (F(X,Z), \frac{\psi(X)}{\sigma^2}  \Big) \Big],
\end{align*}
where $(X,X')\stackrel{i.i.d.}{\sim}U[0,1]$ and $(Z,Z')\stackrel{i.i.d.}{\sim}N(0,1)$ are mutually independent.
\end{definition}

\begin{theorem}
\label{thm:var_conv_new} 
Suppose ${\bf y}\sim N\Big({\mathbf X}{\bm \beta}_0,\sigma^2{\bf I}\Big)$. Writing ${\mathbf X}^{\mathrm T}{\mathbf X}=A_p+D_p$, set $\mathbf{D}$ to denote the vector in $\mathbb{R}^p$ containing the diagonal entries of $D_p$.  Assume the following:

\begin{enumerate}
\item[(i)]$d_\square(W_{pA_p},W)\to 0$ for some $W\in \mathcal{W}$.

\item[(ii)] $ w_{\bbeta_0 }\stackrel{L^1}{\rightarrow} \phi(\cdot)$, and  $w_{{\bf D} }\stackrel{L^1}{\rightarrow} \psi(\cdot)$.

\item[(iii)] ${\rm tr}(A_p^2)=o(p)$.
\end{enumerate}
Define a function $g\in L^1[0,1]$ by $$g(x)=\int_{[0,1]} W(x,y) \phi(y) \mathrm{d}y+ \psi(x) \phi(x).$$ 
Then we have, 
\begin{align}
&\frac{1}{p}\sup_{{\bf u}\in [-1,1]^p}M_p({\bf u} ) \stackrel{P|\mathbf{X}}{\longrightarrow} \sup_{F\in \mathscr{F}} \mathcal{G}_{W,g,\psi}(F). \label{eq:var_prob} 
\end{align}
\end{theorem}
\noindent
Lemma \ref{lem:off_diagonal} and Theorem \ref{thm:var_conv_new} are established in Section \ref{subsec:thm2_proof}.  


\subsection{Uniqueness of the optimizer} 
Theorem \ref{thm:main} identifies general conditions for the asymptotic tightness of the naive mean-field lower bound to the log-normalizing constant. The Gibbs Variational Principle \cite{wainwright2008graphical} establishes that under these settings, the distribution $\mu$ can be approximated, to the leading order, by a product distribution. However, this does not specify whether the ``best" approximation is unique. Indeed, the ferromagnetic Ising model on the complete graph, henceforth referred to as the Curie-Weiss model, provides a classical example where the mean-field lower bound is tight, but without an external magnetization, the model has two distinct optimizers at ``low temperature". 

In this section, we identify some conditions which guarantee the uniqueness of the minimizers. From a statistical perspective, these conditions allow us to conclude that the posterior distribution roughly behaves like a product distribution. To this end, our next lemma identifies a set of sufficient conditions. 
\begin{definition}\label{def:ln}
Let $(\xi_1,\ldots,\xi_p)\stackrel{i.i.d.}{\sim}N(0,1)$ be as constructed in Lemma \ref{lem:off_diagonal}. For any $\mathbf{u} \in [-1,1]^p$, define a random empirical measure $L_p^{(\mathbf{u})} $ on $\R^3$ by setting
$$L_p^{(\mathbf{u})} :=\frac{1}{p} \sum_{i=1}^{p} \delta_{(\frac{i}{p}, \xi_i, u_i)}.$$
\end{definition}
\begin{theorem}
\label{thm:pure_state} 
Suppose all assumptions of Theorem \ref{thm:var_conv_new} hold. Assume further that there exists ${\bf u}_p^*\in [-1,1]^p$ such that for all $\varepsilon>0$ there exists $\delta'>0$ such that 
\begin{align}\label{eq:separate}
\mathbb{P}\Big[\sup_{{\bf u}:\lVert {\bf u}-{\bf u}^*_p\rVert_2^2>p\varepsilon}\frac{1}{p}\{M_p({\bf u})-M_p({\bf u}^*_p)\}<-\delta' | \mathbf{X} \Big] = 1- o(1).
\end{align}
Then the following conclusions hold. 
\begin{itemize} 
\item[(i)]  The limiting variational problem \eqref{eq:var_prob} has a unique optimizer $F^* \in \mathscr{F}$. 
\item[(ii)] $L_p^{(\mathbf{u}_p^*)}\stackrel{P}{\to}  \mu^*,$
 where $\mu^*$ is the law of $(X,Z,F^*(X,Z))$ with $X\sim U[0,1]$ and $Z\sim N(0,1)$ mutually independent. 
 \item[(iii)]  $F^*$ satisfies the fixed point equation:
\begin{align}\label{eq:rde}
F^*(x,z) \stackrel{a.s.}{=} \dot{c}\Big( \frac{1}{\sigma^2} \Big( - \mathbb{E}[W(x, X) F^*(X,Z)] +g(x) + \sqrt{\psi(x)} z  \Big), \frac{\psi(x)}{\sigma^2} \Big), 
\end{align}
where $X\sim U[0,1]$ and $Z\sim N(0,1)$ are mutually independent.
\end{itemize} 

\end{theorem}

Combining Theorem \ref{thm:main} and Theorem \ref{thm:pure_state}, we get the following corollary, which deduces a Law of Large numbers under the posterior distribution.

\begin{corollary}\label{cor:eg}
Suppose \eqref{eq:weak}, \eqref{eq:separate}, and the three conditions (a), (b), (c) of Theorem \ref{thm:var_conv_new} hold.  Then the optimization in Theorem \ref{thm:var_conv_new} has a unique solution $F^*\in \mathscr{F}$ that satisfies
$$F^*(x,z)\stackrel{a.s.}{=}\dot{c}\Big( \frac{1}{\sigma^2} \Big( - \mathbb{E}[W(x, X) F(X,Z)] +g(x) + \sqrt{\psi(x)} z  \Big), \frac{\psi(x)}{\sigma^2} \Big), 
$$
 Further, for any continuous function $\zeta:[-1,1]\times [0,1]\mapsto \R$ we have
$$\mu_{\mathbf{y}, \mathbf{X}}\Big(  \Big| \frac{1}{p}\sum_{i=1}^p \zeta(\beta_i, \beta_{0,i}) - \int_{[-1,1] \times [0,1]} \zeta(w,x) d\pi_{ h \big(F^*(x,z),\frac{\psi(x)}{\sigma^2} \big),\frac{\psi(x)}{\sigma^2}}(w) dx  \Big| > \varepsilon\Big)  \stackrel{P|\mathbf{X}}{\to} 0.$$

\end{corollary}
\begin{remark}
To use this corollary in concrete examples, one needs to compute the functions $(W,g,\psi)$, and verify the conditions \eqref{eq:mean_field}, \eqref{eq:weak} and \eqref{eq:separate}. Of these, the separation condition \eqref{eq:separate} is somewhat implicit, while the other two conditions are relatively easy to verify directly. The following lemma provides two sufficient conditions to this end. In particular, it shows that \eqref{eq:separate} holds in the so called \enquote{high temperate regime}, or if $\pi$ has a density (with respect to Lebesgue measure) which is log concave.
\end{remark}
\begin{lemma}
\label{lemma:pure_state} 
\begin{enumerate}
\item[(i)]
Suppose there exists $\lambda>0$ and ${\bf u}_p^*\in [-1,1]^p$ such that for all ${\bf u}\in [-1,1]^p$ we have
\begin{align}\label{eq:ps}
\mathbb{P}\Big[ M_p({\bf u}^*)-M_p({\bf u})\ge \lambda \lVert {\bf u}-{\bf u}^*\rVert_2^2  \Big| \mathbf{X} \Big] = 1- o(1),
\end{align}
for some $\lambda>0$, free of ${\bf u}$ and $p$. 
Then the condition \eqref{eq:separate} in Theorem \ref{thm:pure_state} holds.

\item[(ii)]
In particular \eqref{eq:ps} holds under either of the following conditions:

\begin{itemize}
\item[(a)]
$$
\limsup_{p \to \infty} \sup_{1\le i\le p} \sum_{j\ne i} |A_p(i,j)| <\sigma^2. \nonumber
$$
\item[(b)] Let the prior $\pi$ have a density with respect to Lebesgue measure on $[-1,1]$, 
\begin{align}
\frac{\mathrm{d}\pi}{\mathrm{d}x} = \frac{1}{Z} \exp(-V(x)), \nonumber 
\end{align}
where $V$ is even, $V:[0,1] \to \mathbb{R}$ is increasing, and $V'$ is convex on $[0,1)$. Further, we assume that $\liminf_{p \to \infty} \lambda_{\min}(\mathbf{X}^{\mathrm{T}}\mathbf{X}) >0$. 
\end{itemize}
\end{enumerate}
\end{lemma}

\noindent
We collect the proofs of Theorem \ref{thm:pure_state} and Corollary \ref{lemma:pure_state} in Section \ref{subsec:proof_thm3}.


 
\subsection{Applications} 
To illustrate the utility of Theorem \ref{thm:var_conv_new} and Corollary \ref{cor:eg}, we apply our results to specific examples in this section. The proofs are deferred to Appendix \ref{sec:example_proofs}. 

\noindent 
\textbf{Spiked Covariance Matrix:} We consider linear regression with mean-zero gaussian features. We assume a spike covariance structure \cite{johnstone2009consistency} on the features. 
\begin{corollary}\label{cor:reg2}
Suppose that the $i^{th}$ row of the design matrix ${\mathbf X}$ equals $n^{-1/2}  {\mathbf x}_i$, where $\{\mathbf{x}_i\}_{1\le i\le n}\stackrel{i.i.d.}{\sim} N(0,\Gamma_p)$, where $\Gamma_p={\bf I}+{\bf v}{\bf v}'$, with $v_i:=\frac{1}{\sqrt{p}} G(i/p)$, and $G:[0,1]\mapsto \R$ is continuous almost surely.
%
%
%
%
%
Further assume that the true regression coefficient $\bbeta_0$ satisfies
$w_{\bbeta_0}\stackrel{L^1}{\to} \phi$, for some $\phi\in L^1[0,1]$.

\begin{enumerate}
\item[(a)]
Then for any $\varepsilon>0$, 
 $$  \frac{1}{p}\log Z_p({\bf y},{\bf X}) \stackrel{P{| \mathbf{X}}}{\longrightarrow}  \sup_{F\in \mathcal{F}}\mathcal{G}_{W, g,\psi}(F)    ,$$
 where $$W(x,y)=G(x) G(y) \quad  \psi(x)=1,\quad 
g(x)=G(x)\int_{[0,1]}G(y)\phi(y)dy+\phi(x).$$

\item[(b)]
Consider the following two cases: either (i) $\lambda<\sigma^2$, or (ii) the conditions of Lemma \ref{lemma:pure_state} Part (b) (ii) hold. 
%
%
%
%
%
 Then the conclusions of Corollary \ref{cor:eg} hold.

\end{enumerate}
\end{corollary}
The next example re-visits the sparse bernoulli design setting introduced in Corollary \ref{cor:erdos1}. 
\begin{corollary}\label{cor:erdos2}

Suppose that the $(i,j)^{th}$ entry of the design matrix ${\mathbf X}$ equals $\sqrt{\frac{p}{n}} B(i,j)$, where $\{B(i,j)\}_{1\le i\le n,1\le j\le p}$ are mutually independent Bernoulli random variables, with
$\P(B(i,j)=1)=\frac{1}{p}G\Big(i/n,j/p\Big)$, where $G$ is a function on $[0,1]^2$ which is continuous almost surely.
%
%
%
%
%
%
Further assume that the true regression coefficient $\bbeta_0$ satisfies
$w_{\bbeta_0}\stackrel{L^1}{\to} \phi$, for some function $\phi\in L^1[0,1]$.

\begin{enumerate}
\item[(a)]
Then
 $$\frac{1}{p}\log Z_p({\bf y},{\bf X})\stackrel{P}{\to}\sup_{F\in \mathcal{F}}\mathcal{G}_{W, g,\psi}(F),$$
 where
 \begin{align*}
 W(x,y)=\int_{[0,1]} G(t,x) G(t,y) dt,\quad \psi(x)=\int_{[0,1]}G(t,x)dt,\quad
g(x)=\int_{[0,1]}W(x,y) \phi(y)dy+\phi(x)\psi(x).
\end{align*}

\item[(b)]
Consider the following two cases:

\begin{enumerate}
\item[(i)]
Set $$S(x):=\int_{[0,1]}W(x,y)dy=\int_{[0,1^2]} G(t,x) G(t,y) dt dy.$$
 Suppose $\sigma^2>{\rm ess} \,  \sup S(X)$, where $X\sim U[0,1]$;

\item[(ii)]
the conditions of Lemma \ref{lemma:pure_state} Part (b) (ii) hold.

\end{enumerate}

\noindent 
 Then the conclusions of Corollary \ref{cor:eg} hold.

%

\end{enumerate}
\end{corollary}

\noindent
As our last example, we consider a sequence of design matrices arising in the study of two-way ANOVA designs. 
\begin{corollary}\label{cor:det2}

Suppose that we have a two factor ANOVA model of the form
\begin{align*}
y_{ij}=\frac{1}{\sqrt{p}}(\tau_i+\gamma_j)+\xi_{ij}, 1\le i,j\le \tilde{p}
\end{align*}
Here $(\tau_1,\ldots,\tau_{\tilde{p}})\in [-1,1]^{\tilde{p}}$ are the levels of the first factor, and $(\gamma_1,\ldots,\gamma_{\tilde{p}})\in [-1,1]^{\tilde{p}}$ are the levels of the second factor, and $\xi_{ij}\stackrel{i.i.d.}{\sim}N(0,\sigma^2)$. Setting $n=\tilde{p}^2$ and $p=2\tilde{p}$, let ${\bf y}\in \mathbb{R}^n$ denote the vector obtained by linearizing the matrix $((y_{ij}))$ row-wise, and $\mathbf{X}$ denote the corresponding $n\times p$ design matrix, and let $\bbeta:=(\tau,\gamma)\in [-1,1]^p$ be the unknown parameter.

Assume that the true regression coefficient $\bbeta_0\in [-1,1]^p$ satisfies
$w_{\bbeta_0}\stackrel{L^1}{\to} \phi$, for some function $\phi\in L^1[0,1]$.

\begin{enumerate}
\item[(a)]
Then
 $$\frac{1}{p}\log Z_p({\bf y},{\bf X})\stackrel{P}{\to}\sup_{F\in \mathcal{F}}\mathcal{G}_{W, g,\psi}(F),$$
 where $\psi = \frac{1}{2}$, 
 \begin{align*}
W(x,y):=&0 \text{ if }(x,y)\in [0,.5)^2 \cup (.5,1]^2,\\
=&1\text{ if }(x,y)\in [0,.5)\times (.5,1]\cup  (.5,1]\times [0,.5).
\end{align*}
and 
$g(x)=\int_{[0,1]}W(x,y) \phi(y)dy+\frac{1}{2}\phi(x)$.

\item[(b)]
Consider the following two cases:

\begin{enumerate}
\item[(i)]
$\sigma^2>\frac{1}{2}$.

\item[(ii)]
the conditions of Lemma \ref{lemma:pure_state} Part (b) (ii) hold.

\end{enumerate}
\noindent
 Then the conclusions of Corollary \ref{cor:eg} hold.

%

\end{enumerate}
\end{corollary}

\section{Discussions} 
\label{sec:discussions} 
We discuss some limitations of our current results, and collect some questions for future enquiry. 
\begin{itemize}
\item[(i)] The bounded support assumption on the prior---A vital technical assumption in our analysis concerns the bounded support assumption on the prior. We note that for a general prior with unbounded support, the posterior $\mu_{\mathbf{y}, \mathbf{X}}$ might not even be a proper probability distribution. One intuitively expects that under appropriate ``tail-decay" conditions on the prior, the results in this paper should generalize. Going beyond the bounded support assumption requires extending the theory of non-linear large deviations to probability measures on unbounded spaces, and is thus beyond the scope of this paper. 

\item[(ii)] Extensions to Gibbs posteriors and fractional posteriors---A careful study of our proof reveals that our main results do not depend strongly on the correct model specification. As a result, we expect similar techniques to be broadly useful in the study of Gibbs and fractional posteriors  \cite{alquier2016properties,yang2020alpha}. 

\item[(iii)] Extensions to other GLMs---Another natural question of interest concerns the applicability of these ideas to more general models, e.g. logistic regression. We consider this to be an extremely interesting question, and plan to explore this in the future. 

\item[(iv)] Extensions to models with latent characteristics---Variational methods are ubiquitous in applications with latent characteristics e.g. topic modeling, community detection etc. In contrast, the relevant variables are all observed in the linear model framework. It will be interesting to explore the applicability of our techniques to models with latent features. 
\end{itemize} 

\section{Proofs} 
\label{sec:proofs} 
We prove the main results in this section. Theorem \ref{thm:main} is established in Section \ref{subsec:proof_thm1}, Theorem \ref{thm:var_conv_new} is established in Section \ref{subsec:thm2_proof}, while Theorem \ref{thm:pure_state} is proved in Section \ref{subsec:proof_thm3}. 

\subsection{Proof of Theorem \ref{thm:main}}
\label{subsec:proof_thm1} 
Our first  lemma collects some basic facts about the exponential family $\psi_{\gamma}$. The proof is deferred to the Appendix. 

\begin{lemma}
\label{lemma:elementary} 
In the setting of Lemma \ref{lemma:properties}, we have the following conclusions. 
\begin{itemize}
\item[(i)]  Set $\gamma= (\gamma_1, \gamma_2) \in \mathbb{R}^2$. The derivatives of $c(\gamma_1,\gamma_2)$ with respect to $\gamma_1$ are given by 
\[\dot{c}(\gamma):=\frac{\partial c(\gamma_1,\gamma_2)}{\partial \gamma_1}=\int_{[-1,1]} z \mathrm{d}\pi_{\gamma}(z),\quad \ddot{c}(\gamma):=\frac{\partial^2 c(\gamma_1,\gamma_2)}{\partial \gamma_1^2}=\int_{[-1,1]} (z-\dot{c}(\gamma))^2 \mathrm{d}\pi_\gamma(z).\]
\item[(ii)] We have, for $\gamma= (\gamma_1, \gamma_2), \tilde{\gamma}= (\tilde{\gamma}_1, \gamma_2) \in \mathbb{R}^2$,  \[ \dkl (\pi_{\gamma} \| \pi_{(0,\gamma_2)})=\gamma_1 \dot{c}(\gamma)-c(\gamma) + c(0,\gamma_2) .\]

\item[(iii)] Consider a sequence $\gamma_k = (\gamma_{1,k}, \gamma_{2,k}) \in \mathbb{R}^2$ such that $\gamma_{1,k} \to \pm\infty$ and $\limsup |\gamma_{2,k}| < \infty$. Then $\pi_{\gamma_k} \stackrel{w}{\to} \pi_{\pm \infty}$, where $\pi_{\pm \infty}$ is the degenerate distribution which puts mass 1 at $\pm 1$. 

\item[(iv)] For any $d>0$, the function $G(\cdot, d): [-1,1] \to \mathbb{R}^+$ is lower semicontinuous. 

\item[(v)] For $x \in [-1,1]$, $y \in \mathbb{R}$ and $r \in \mathbb{N}$,  define 
\begin{align}
H(x,y) = \int_{[-1,1]} z^r \mathrm{d}\pi_{(h(x,y),y)}(z) \label{eq:H_defn} 
\end{align} 
Then we have, 
\begin{align} 
&\sup_{x \in (-1,1), y \in \mathbb{R}} \Big| \frac{\partial H (x,y)}{\partial y}  \Big| \leq 2. \label{eq:H_stability}
\end{align} 
\end{itemize}
\end{lemma}

\noindent
We now turn to the proof of Theorem \ref{thm:main}. To this end, set 
\begin{align}
m_i(\bbeta) := \sum_{j=1}^{p} A_{ij} \beta_j, \,\,\,\,\, \theta_i := \frac{z_i - m_i(\bbeta)}{\sigma^2}, \label{eq:local_fields}
\end{align}
where we recall that $\sigma^2$ denotes the noise variance in our linear regression model \eqref{eq:model}, and $\mathbf{z} = \mathbf{X}^{\mathrm{T}} \mathbf{y}$. Observe that $A_{ii}=0$ for all $1\leq i \leq p$ implies that $\mu(\cdot | (\beta_j)_{j \neq i} ) = \pi_{(\theta_i,d_i)}$, and thus $\mathbb{E}_{\mu}[\beta_i | (\beta_j)_{j \neq i} ] = \dot{c}(\theta_i,d_i)$. We define 
\begin{align}
\mathbf{b} = (\dot{c}(\theta_i,d_i))_{ 1\leq i \leq p}.  \label{eq:b_defn}
\end{align}
We will  prove that  certain statistics under the posterior distribution $\mu$ can be ``well-approximated" by the vector of conditional means. To this end, we establish the following results. 

\begin{lemma}
\label{lemma:taylor1} 
Under the conditions of Theorem \ref{thm:main}, 
setting $f(\bbeta) = -\frac{1}{2\sigma^2} \bbeta^{\mathrm{T}} A \bbeta + \frac{1}{\sigma^2} \mathbf{z}^{\mathrm{T}} \bbeta$, we have,
\begin{align}\label{eq:claim1}
\mathbb{E}_{\mu}\Big[\bbeta^{\mathrm{T}}A\bbeta -{\bf b}^{\mathrm{T}}A{\bf b}\Big]^2&=o(p^2),\\
\label{eq:claim2}\mathbb{E}_{\mu}\Big[\sum_{i=1}^pm_i(\bbeta)(\beta_i-b_i)\Big]^2&=o(p^2),\\
\label{eq:claim3}\mathbb{E}_{\mu}\Big[ f(\bbeta)   - f(\mathbf{b}) -\sum_{i=1}^p \theta_i (\beta_i-b_i)\Big]^2&=o(p^2), 
\end{align}
In the above displays, the random vectors $\mathbf{\btheta}:=(\theta_1,\cdots,\theta_p)^{\mathrm{T}}$ and $\mathbf{b}$ are as defined in \eqref{eq:local_fields} and \eqref{eq:b_defn} respectively. 
\end{lemma}


\begin{lemma}\label{lem:cond_center}
Suppose $\phi:[-1,1]\mapsto \R$ is a bounded measurable function, and $A_p$ satisfies \eqref{eq:mean_field} and \eqref{eq:weak}. Then for any ${\bf c}\in [-1,1]^p$ we have
$$\limsup_{p\to\infty}\frac{1}{p}\E_{\mu} \left[\sum_{i=1}^pc_i\Big\{\phi(\beta_i)-\E_{\mu} [\phi(\beta_i)|\beta_k, k\ne i]\Big\}\right]^2 <\infty.$$
\end{lemma}

\noindent 
We defer the proofs of these lemmas to the end of this section, and establish Theorem \ref{thm:main}, given these lemmas. 

\begin{proof}[Proof of Theorem \ref{thm:main}]

\textit{Proof of Part (i). } 
Since $\text{tr}(A^2)=o(p)$, \cite[Theorem 4]{yan} implies 
\begin{align}
\log Z_p- \sup_{Q= \prod_{i=1}^{p} Q_i }\Big[  - \frac{1}{2\sigma^2} ( \mathbb{E}_{Q}[\mathbf{X}])^{\mathrm{T}} A ( \mathbb{E}_{Q}[\mathbf{X}])+ \frac{1}{\sigma^2} \mathbf{z}^{\mathrm{T}} \mathbb{E}_{Q}[\mathbf{X}]-\sum_{i=1}^pD(Q_i \| \pi_i)\Big] = o(p), \nonumber 
\end{align}
where $\mathbb{E}_Q[\mathbf{X}]$ is the mean vector of $\mathbf{X} \sim Q = \prod_{i=1}^{p} Q_i$. The desired conclusion then follows upon using Lemma \ref{lemma:naive_approx}.
\\ 

\noindent 
\noindent 
\textit{Proof of Part (ii).} 
With $f(\bbeta) = -\frac{1}{2\sigma^2} \bbeta^{\mathrm{T}} A \bbeta + \frac{1}{\sigma^2} \mathbf{z}^{\mathrm{T}} \bbeta$ as in Lemma \ref{lemma:taylor1}, define 
\begin{align}
C_p(\varepsilon)  = \Big\{ \mathbf{u} \in [-1,1]^{p} : f(\mathbf{u}) - \sum_{i=1}^p \Big(u_i h(u_i,d_i) -c(h(u_i,d_i),d_i) \Big) < \sup_{\mathbf{u} \in [-1,1]^p} M_p(\mathbf{u}) - p \varepsilon \Big\}. \nonumber 
\end{align} 
Note that it suffices to establish that for all $\varepsilon>0$, $\mu({\bf b}\in C_p(\varepsilon)) \to 0$ as $p \to \infty$. To this end,  define the event 
\begin{align}
E_p(\varepsilon/2) = \Big\{ \bbeta \in [-1,1]^p:  | f(\bbeta) - f(\mathbf{b}) - \sum_{i=1}^{p} \theta_i (\beta_i - b_i) | \leq \frac{p \varepsilon}{2} \Big\}. \nonumber 
\end{align}
Note that \eqref{eq:claim3}, in combination with Chebychev inequality, implies that $\mu(E_p(\varepsilon/2)^{c}) \to 0$ as $p \to \infty$. Therefore, 
\begin{align}
\mu( C_p(\varepsilon)) \leq \mu(C_p (\varepsilon) \cap E_p (\varepsilon/2)) + \mu(E_p(\varepsilon/2)^{c}) = \mu(C_p (\varepsilon) \cap E_p (\varepsilon/2)) + {o(1)} .  \label{eq:prob_bound_int1} 
\end{align} 
Now, we have, 
\begin{align}
\mu(C_p (\varepsilon) \cap E_p (\varepsilon/2)) &= \frac{1}{Z_p } \int_{C_p (\varepsilon) \cap E_p (\varepsilon/2)} \exp(f(\bbeta)) \prod_{i=1}^p \pi_i(\mathrm{d} \beta_i)\nonumber \\
&\leq \frac{\exp( p\varepsilon/2)}{Z_p } \int_{C_p (\varepsilon) \cap E_p (\varepsilon/2)} \exp\Big[ f(\mathbf{b}) + \sum_{i=1}^{p} \theta_i (\beta_i - b_i) \Big]  \prod_{i=1}^p \pi_i(\mathrm{d} \beta_i)\nonumber \\ 
&\leq \frac{\exp( - p\varepsilon/2 + \sup_{\mathbf{u} \in [-1,1]^p} M_p(\mathbf{u}) )}{Z_p } \int_{C_p (\varepsilon) \cap E_p (\varepsilon/2)} \exp\Big[ \sum_{i=1}^{p} (\theta_i \beta_i - c(\theta_i,d_i)) \Big] \prod_{i=1}^p \pi_i(\mathrm{d} \beta_i) , \nonumber 
\end{align}
where the first inequality follows using the definition of $E_p(\varepsilon/2)$, while the last inequality follows from the definition of $C_p(\varepsilon)$. Using the display above in combination with \eqref{eq:mf_conclusion}, it suffices to establish that 
\begin{align}
\log \int_{C_p (\varepsilon) \cap E_p (\varepsilon/2)} \exp\Big[ \sum_{i=1}^{p} (\theta_i \beta_i - c(\theta_i,d_i)) \Big]  \prod_{i=1}^p \pi_i(\mathrm{d} \beta_i)= o(p).   \nonumber 
\end{align}
 This argument would be relatively straight-forward if the $\theta_i$ were fixed constants independent of $\bbeta$, as 
 \begin{align}
 \int_{C_p (\varepsilon) \cap E_p (\varepsilon/2)} \exp\Big[ \sum_{i=1}^{p} (\theta_i \beta_i - c(\theta_i,d_i)) \Big]  \prod_{i=1}^p \pi_i(\mathrm{d} \beta_i) \leq \int_{[-1,1]^p} \exp\Big[ \sum_{i=1}^{p} (\theta_i \beta_i - c(\theta_i,d_i)) \Big]  \prod_{i=1}^p \pi_i(\mathrm{d} \beta_i) = 1. \nonumber 
 \end{align} 
 We caution the reader that this is not the case, and the $\theta_i$ are themselves functions of $\bbeta$, as specified in \eqref{eq:local_fields}. 

To overcome this issue, we proceed as follows. Fix $\delta>0$, and let $\mathscr{D}_p(\delta)$ be a $\sqrt{p} \delta$ net of the set $\{A \bbeta, \bbeta \in [-1,1]^{p}\}$ in the euclidean metric, satisfying $\lim_{p \to \infty} \frac{1}{p} \log | \mathscr{D}_p(\delta) | =0$. Under the assumption $\mathrm{tr}(A^2) = o(p)$, such a net was constructed in \cite[Lemma 3.4]{basakmukherjee}. Thus for every $\bbeta \in [-1, 1]^{p}$, there exists $\mathbf{p} \in \mathscr{D}_p(\delta)$ such that $\| A \bbeta - \mathbf{p} \|_2 \leq \delta \sqrt{p}$. For any $\mathbf{p} \in \mathscr{D}_p(\delta)$, we set 
\begin{align}
\mathcal{P}(\mathbf{p}) = \{ \bbeta \in [-1,1]^{p}: \| A \bbeta - \mathbf{p} \|_2 \leq \delta \sqrt{p} \}. \nonumber 
\end{align}
For $\bbeta\in\mathcal{P}(\mathbf{p})$, setting $\theta_i^{\mathbf{p}} = \frac{z_i - p_i}{\sigma^2}$, we have, 
\begin{align}
\sum_{i=1}^{p} (\theta_i - \theta_i^{\mathbf{p}})^2 = \frac{1}{\sigma^4} \sum_{i=1}^{p} (m_i (\bbeta) - p_i)^2= \frac{1}{\sigma^4} \| A \bbeta - \mathbf{p} \|_2^2 \leq \frac{p\delta}{\sigma^4}, \nonumber 
\end{align}
where the last inequality uses the definition of $\mathcal{P}(\mathbf{p})$. This gives, 
\begin{align}
\Big| \sum_{i=1}^{p} (\theta_i - \theta_i^{\mathbf{p}} ) \beta_i  \Big| \leq \frac{p \sqrt{\delta}}{\sigma^2},\quad &
\Big| \sum_{i=1}^{p} \Big(c(\theta_i,d_i) - c(\theta_i^{\mathbf{p} },d_i ) \Big) \Big| \leq \sum_{i=1}^{p} | \theta_i - \theta_i^{\mathbf{p}} | \leq \frac{p \sqrt{\delta}}{\sigma^2}. \nonumber 
\end{align}
where the first inequality follows from Cauchy-Schwarz, and the second inequality follows from the smoothness of $c(\cdot, \cdot)$. 
Thus we have, 
\begin{align}
&\int_{C_p (\varepsilon) \cap E_p (\varepsilon/2)} \exp\Big[ \sum_{i=1}^{p} (\theta_i \beta_i - c(\theta_i,d_i)) \Big]  \prod_{i=1}^p \pi_i(\mathrm{d} \beta_i)\nonumber \\
&\leq  \sum_{\mathbf{p} \in \mathscr{D}_p(\delta)} \int_{\mathcal{P}(\mathbf{p}) } \exp\Big[ \sum_{i=1}^{p} (\theta_i \beta_i - c(\theta_i,d_i)) \Big]\prod_{i=1}^p  \pi_i(\mathrm{d}\beta_i). \nonumber \\
&\leq \exp \Big( \frac{ 2p \sqrt{\delta}}{\sigma^2} \Big) \int_{[-1,1]^{p} } \exp\Big[ \sum_{i=1}^{p} (\theta_i^{\mathbf{p}}  \beta_i - c(\theta_i^{\mathbf{p}},d_i )) \Big]  \prod_{i=1}^p  \pi_i(\mathrm{d}\beta_i)= \exp \Big( \frac{ 2p \sqrt{\delta}}{\sigma^2} \Big) .  \nonumber
\end{align}
This concludes the proof, as $\delta>0$ is arbitrary. 
\\

\noindent 
\textit{Proof of Part (iii).} 
First, note that it suffices to show the result with $\zeta(x,y) = x^r y^s$ for any $r,s \in \mathbb{N}$. Thus we fix $r, s \in \mathbb{N}$ in the rest of the proof. 
Using Lemma \ref{lem:cond_center} with $\phi(x)=x^r$ and $c_i=(\frac{i}{p})^s$ it follows that for any $\varepsilon>0$, 
$$\mu \Big( \Big|  \frac{1}{p}\sum_{i=1}^p \beta_i^r \Big(\frac{i}{p}\Big)^s-\frac{1}{p}\sum_{i=1}^p H(b_i,d_i)\Big(\frac{i}{p}\Big)^s  \Big| > \varepsilon \Big) \to 0,$$
where $H$ is defined as in \eqref{eq:H_defn}.
To complete the proof, it thus suffices to show that 
\begin{align}\label{eq:2021_1}
\mu \Big( \Big|  \frac{1}{p}\sum_{i=1}^p H(b_i,d_i)\Big(\frac{i}{p}\Big)^s-\frac{1}{p}\sum_{i=1}^p H(\hat{u}_i,d_i)\Big(\frac{i}{p}\Big)^s \Big| > \varepsilon  \Big) \to 0. 
\end{align}
Fixing $ K<\infty$ and setting $d_K(i):=d_i 1\{|d_i|\le K\}$, using Lemma \ref{lemma:elementary} Part (vi)  we have
\begin{align}\label{eq:2021_2}
\frac{1}{p}\left|\sum_{i=1}^p H(\hat{u}_i,d_i)\Big(\frac{i}{p}\Big)^s-\sum_{i=1}^pH(\hat{u}_i,d_K(i))\Big(\frac{i}{p}\Big)^s\right|\le \frac{2}{p}\sum_{i=1}^p d_i1\{|d_i|>K\}.
\end{align}
which goes to $0$ as $p\to\infty$ followed by $K\to\infty$, using the uniform integrability assumption on $\frac{1}{p}\sum_{i=1}^p\delta_{d_i}$.
Also, with $\delta>0$ and setting 
\begin{align*}
\hat{u}_\delta(i):=&\hat{u}_i\text{ if }|\hat{u}_i|\le 1-\delta,\\
:=&1-\delta\text{ if }\hat{u}_i>1-\delta,\\
:=&-1+\delta\text{ if }\hat{u}_i<-1+\delta,
\end{align*}
we have
\begin{align}\label{eq:2021_3}
\frac{1}{p}\left|\sum_{i=1}^p H(\hat{u}_i,d_K(i))\Big(\frac{i}{p}\Big)^s-\sum_{i=1}^p H(\hat{u}_\delta(i),d_K(i))\Big(\frac{i}{p}\Big)^s\right| \le &\sup_{x\in [1-\delta,1], |y|\le K} |H(x,y)-H(1-\delta,y)|+ \nonumber \\
&\sup_{x\in [-1,1+\delta], |y|\le K}|H(x,y)-H(-1+\delta,y)|. 
\end{align}
Fixing $K>0$, we now  claim that the RHS of \eqref{eq:2021_3} converges to $0$ as $\delta\to 0$. Given this claim, it follows from \eqref{eq:2021_2} and \eqref{eq:2021_3} that
$$\limsup_{K\to\infty}\limsup_{\delta\to 0}\limsup_{p\to\infty}\Big|\frac{1}{p}\sum_{i=1}^pH( \hat{u}_i,d_i)\Big(\frac{i}{p}\Big)^s-\frac{1}{p}\sum_{i=1}^pH(\hat{u}_\delta(i),d_K(i))\Big(\frac{i}{p}\Big)^s\Big|=0.$$
A similar argument with $\hat{u}_i$ replaced by $b_i$ gives
$$\limsup_{K\to\infty}\limsup_{\delta\to 0}\limsup_{p\to\infty}\Big|\frac{1}{p}\sum_{i=1}^pH(b_i,d_i)\Big(\frac{i}{p}\Big)^s-\frac{1}{p}\sum_{i=1}^pH(b_\delta(i),d_K(i))\Big(\frac{i}{p}\Big)^s\Big|=0,$$
where $b_\delta(i):=b_i1\{|b_i|\le 1-\delta\}$. 

It suffices to show that for every $\delta,K$ fixed we have 
$$\mu \Big( \Big|\frac{1}{p}\sum_{i=1}^pH(\hat{u}_\delta(i),d_K(i))\Big(\frac{i}{p}\Big)^s-\frac{1}{p}\sum_{i=1}^pH(b_\delta(i),d_K(i))\Big(\frac{i}{p}\Big)^s\Big| > \varepsilon \Big)  \to 0.$$
But this follows on noting that 
$$\sup_{|x|\le 1-\delta, |y|\le K}\Big|\frac{\partial H(x,y)}{\partial x}\Big|
=:M<\infty,$$
and so  
$$\frac{1}{p}\left|\sum_{i=1}^p H(\hat{u}_\delta(i),d_K(i))\Big(\frac{i}{p}\Big)^s- \sum_{i=1}^p H(b_\delta(i),d_K(i))\Big(\frac{i}{p}\Big)^s\right|\le \frac{M}{p}\sum_{i=1}^p | \hat{u}_\delta(i)-b_\delta(i)|\le M\sqrt{\frac{1}{p}\sum_{i=1}^p \Big[\hat{u}_i-b_i \Big]^2},$$
which converges to zero in probability under the posterior distribution $\mu_{\mathbf{y}, \mathbf{X}}(\cdot)$. 
Here the last estimate uses part (ii) of this Theorem.
\\

To complete the argument, it thus remains to verify the claim involving \eqref{eq:2021_3}, for which it suffices to verify continuity of the function $(x,y)\mapsto H(x,y)$ at $x = \pm 1$, uniformly for $y\in [-K,K]$. By symmetry, it suffices to show that
$$\lim_{\delta\to 0} \sup_{x\in [1-\delta,1],|y|\le K}| H(x,y)-1|\to 0.$$

Suppose this is not true. Then there exists sequences $\{x_k\}_{k\ge 1}, \{y_k\}_{k\ge 1}$ with $\lim_{k\to\infty}x_k=1$, and $|y_k|\le K$, such that $|H(x_k,y_k)-1|>\varepsilon$ for all $k$, for some $\varepsilon>0$. Without loss of generality, by passing to a subsequence, we can assume that $y_k$ converges to $y\in [-K,K]$. Using Lemma \ref{lemma:elementary} Part (iii) we have that $\pi_{(h(x_k,y_k),y_k)}$ converges weakly to $\delta_1$, the point mass at $1$. Consequently, using DCT we have 
$$H(x_k,y_k)=\int_{[-1,1]} x^r d\pi_{(h(x_k,y_k),y_k)}(z)\to 1,$$ a contradiction. This verifies the claim, and hence completes the proof of part (iii).

\end{proof}

\noindent 
Next we turn to the proof of Lemmas \ref{lemma:taylor1} and \ref{lem:cond_center}. 

\begin{proof}[Proof of Lemma \ref{lemma:taylor1}]
Since \eqref{eq:claim1}  and \eqref{eq:claim2} follow immediately from \cite[(4.40)]{yan} and \cite[(4.41)]{yan} respectively, it suffices to prove \eqref{eq:claim3}. But this follows upon observing that
\begin{align*}
\Big|f(\bbeta)-f({\bf b})-\sum_{i=1}^p \theta_i(\beta_i-b_i)\Big|
\le  \frac{1}{\sigma^2}\Big|\sum_{i=1}^pm_i(\bbeta)(\beta_i-b_i) \Big|+\frac{1}{2\sigma^2}|\bbeta^{\mathrm{T}} A \bbeta-\mathbf {b}^{\mathrm{T}} A \mathbf{b}|,
\end{align*}
and using \eqref{eq:claim1} and \eqref{eq:claim2}.

\end{proof}

\begin{proof}[Proof of Lemma~\ref{lem:cond_center}]

For any $i,j\in [p]$ with $i\ne j$, setting $$t_i:=\E[\phi(\beta_i)|\beta_k, k\ne i]=\int_{[-1,1]} \phi(z) d\pi_{(\theta_i,d_i)}(z), \quad \theta_i:=\frac{z_i-\sum_{k=1}^pA_{ik}\beta_k}{\sigma^2}$$ we have
\begin{align}\label{eq:cond_center}
\E_{\mu}\Big[\phi(\beta_i)-t_i\Big] \Big[\phi(\beta_j)-t_j \Big]=\E_{\mu}\Big[\phi(\beta_i)-t_i\Big] \Big[\phi(\beta_j)-t_j^{(i)}\Big]+\E_{\mu}\Big[\phi(\beta_i)-t_i\Big] \Big[t_j^i-t_j\Big]
\end{align}
where $$t_j^{(i)}:=\int_{[-1,1]} \phi(z) d\pi_{(\theta_j^i,d_j)}(z),\quad \theta_j^{(i)}:=\frac{z_j-\sum_{k\ne i}A_{jk}\beta_k}{\sigma^2}.$$
Since $t_j^{(i)}$ does not depend on $\beta_i$, we have
\begin{align*}
\E_{\mu}\Big[\phi(\beta_i)-t_i\Big] \Big[\phi(\beta_j)-t_j^{(i)}\Big]=\E_{\mu}\left( \Big[\phi(\beta_j)-t_j^{(i)} \Big] \E_{\mu}\Big[\phi(\beta_i)-t_i\Big|\beta_k, j\ne i\Big] \right)=0.
\end{align*}
For estimating the second term in the RHS of \eqref{eq:cond_center}, setting $$\ell (x):=\int_{[-1,1]} \phi(z) d\pi_{(x,d_i)}(z)$$ a Taylor's series expansion gives
$$t_j^{(i)}-t_j= \ell(\theta_j^{(i)})-\ell(\theta_j)=-\frac{A_{ij}}{\sigma^2}\beta_i \ell'(\theta_j)+\frac{A_{ij}^2}{2\sigma^4}\beta_i^2\ell''(\xi_j^{(i)}).$$

Using the last two displays and summing over $i,j$ give
\begin{align*}
\left|\sum_{i\ne j}c_i c_j\E_{\mu}\Big[\phi(\beta_i)-t_i\Big]\Big[\phi(\beta_j)-t_j\Big] \right|\le & \left|\sum_{i\ne j}c_ic_j\frac{A_{ij}}{\sigma^2} \beta_i  \ell'(\theta_j)\Big[\phi(\beta_i)-t_i\Big]\right|+\left|\sum_{i,j=1}^pc_ic_j\frac{A_{ij}^2}{2\sigma^4}\beta_i^2 \ell''(\xi_j^{(i)})\right|\\
\le &\frac{2\lVert \phi \rVert_\infty }{\sigma^2} \sup_{{\bf u}\in [-1,1]^p}\sum_{i=1}^p \Big|\sum_{j=1}^p A_{ij}u_j\Big|+\frac{1}{2\sigma^4}\sum_{i,j=1}^p A_{ij}^2,
\end{align*}
where the last estimate uses the fact that $\lVert \ell^{(r)}\rVert_\infty\le 1$ for $r= 1,2$. The desired conclusion follows upon using the given hypothesis on $A_p$.

\end{proof}

\subsection{Proof of Theorem~\ref{thm:var_conv_new}}
\label{subsec:thm2_proof} 
We start with a proof of Lemma \ref{lem:off_diagonal}. 

\begin{proof}[Proof of Lemma \ref{lem:off_diagonal}]
Without loss of generality assume that the samples $(Z_1,\cdots,Z_p)$ are arranged in increasing order of variance, i.e. increasing order of $\{D_p(i,i)\}_{i=1}^p$. Let $\{\delta_p : p \geq 1\}$ be a positive sequence converging to $0$, such that 
\[\frac{1}{\delta_p^{4}}\sum_{i, j=1}^{p} A_p(i,j)^2=o(p).\]
The existence of such a sequence follows from the assumption \eqref{eq:mean_field}. 
Let \[q:=\arg\min \{i\in [p]:D_p(i,i)\ge \delta_p\},\quad r:=\arg\max \{i\in [p]:D_p(i,i)\le \frac{1}{\delta_p}\}.\] 
Recall that if $Z \sim \mathcal{N}(0,1)$, $\mathbb{E}|Z| =\sqrt{\frac{2}{\pi}}$. This implies 
\begin{eqnarray}\label{eq:pd}
\begin{split}
&\mathbb{E}\sup_{{\bf u}\in [-1,1]^p}|\sum_{i=1}^{q-1} Z_i u_i|\le \mathbb{E}\sum_{i=1}^{q-1}|Z_i| = \sqrt{\frac{2}{\pi}} \sum_{i=1}^{q-1} \sqrt{D_p(i,i)}\le \sqrt{\frac{2\delta_p}{\pi}} p=o(p),\\
&\mathbb{E}\sup_{{\bf u}\in [-1,1]^p}|\sum_{i=r+1}^p Z_i u_i|\le \mathbb{E}\sum_{i=r+1}^p|Z_i|\le\sqrt{\frac{2}{\pi}} \sum_{i=1}^p \sqrt{D_p(i,i)} \mathbf{1}\Big(D_p(i,i) \geq \frac{1}{\delta_p} \Big)\le \sqrt{\frac{2\delta_p}{\pi} }\sum_{i=1}^p D_p(i,i)=o(p),
\end{split}
\end{eqnarray}
where we use $\sum_{i=1}^{p} D_p(i,i) = O(p)$. 
For $q\le i\le r$, setting $Y_i:=\frac{Z_i}{\sqrt{D_p(i,i)}}$,  
let  
 $\Gamma_p$ denote the $(r-q+1\times r-q+1)$ dimensional covariance matrix of the vector ${\bf Y}:=(Y_{q},\cdots,Y_r)^\top$.
Define $\mathcal{E}=(\xi_{q},\cdots,\xi_{r})^{\top}:=\Gamma_p^{-1/2}{\bf Y}$, and note that $(\xi_q,\cdots,\xi_r)\stackrel{i.i.d.}{\sim} \mathcal{N}(0,1)$. 
Setting  $v_i:=u_i\sqrt{D_p(i,i)}$ and ${\bf v}:=(v_q,\cdots,v_r)$, we have
\begin{align}\label{eq:represent}\sum_{i=q}^ru_iZ_i=\sum_{i=q}^r v_iY_i={\bf v}^{\top}\Gamma^{1/2}\mathcal{E}={\bf v}^\top \mathcal{E}+{\bf v}^\top B\mathcal{E}=\sum_{i=q}^r \xi_i \sqrt{D_p(i,i)}u_i+{\bf v}^\top B\mathcal{E},
\end{align}
where $B:=\Gamma_p^{1/2}-{\bf I}_{r-q+1}$. 
We now claim that
\begin{align}\label{eq:vdelta}
\frac{1}{p}\sup_{{\bf v}\in \frac{1}{\delta_p}[-1,1]^{r-q+1}}{\bf v}^\top B\mathcal{E}\stackrel{P}{\rightarrow}0.
\end{align}
Given \eqref{eq:vdelta}, it follows from \eqref{eq:represent} that
\begin{align}\label{eq:vdelta2}
\frac{1}{p}\sup_{{\bf u}\in [-1,1]^{r-q+1}}\Big|\sum_{i=q}^r u_iZ_i-\sum_{i=q}^r u_i\sqrt{D_p(i,i)} \xi_i\Big|\stackrel{P}{\rightarrow}0.
\end{align}
Finally with $\{\xi_i\}_{[p]\backslash [q,r]}\stackrel{i.i.d.}{\sim} \mathcal{N}(0,1)$ independent of $\{\xi_i\}_{i=q}^r$, using arguments similar to the derivation of \eqref{eq:pd} we get
\begin{align}
\begin{split}
\label{eq:pd2}
\mathbb{E}\sup_{{\bf u}\in [-1,1]^p}\Big|\sum_{i=1}^{q-1}u_i\sqrt{D(i,i)}\xi_i\Big|\le &\sqrt{\frac{2\delta_p}{\pi}} p=o(p)\\
\mathbb{E}\sup_{{\bf u}\in [-1,1]^p}\Big|\sum_{i=r+1}^{p}u_i\sqrt{D(i,i)}\xi_i\Big|\le &\sqrt{\frac{2\delta_p}{\pi}} \sum_{i=1}^pD_p(i,i)=o(p).
\end{split}
\end{align}
Combining \eqref{eq:pd}, \eqref{eq:vdelta2} and \eqref{eq:pd2} the desired conclusion follows.
\\

It thus remains to verify \eqref{eq:vdelta}. To this effect, note that $\Gamma_p(i,i)=1$, and $\Gamma_p(i,j)=\frac{A_p(i,j)}{\sqrt{D_p(i,i)D_p(j,j)}}$. Setting $C_\delta(i,j):=\Gamma_p(i,j)1\{i\ne j\}$, and using Spectral Theorem, we have,  \[C_\delta=\sum_{\ell=q}^r \lambda_\ell \psi_\ell \psi_\ell^{\top}=\Psi \Lambda \Psi^{\top}.\]
Observe that $\Gamma_p$ is positive semidefinite, and has eigenvalues $1 + \lambda_{\ell}$ with $\lambda_{\ell} \geq -1$. 
Then we have  
\begin{align}\label{eq:mean_field_A}
\sum_{\ell=q}^r\lambda_\ell^2=\sum_{i,j=1}^pC_\delta(i,j)^2\le \frac{1}{\delta_p^{2 }}\sum_{i, j=1}^{p} A_p(i,j)^2.
\end{align}
Finally, the matrix $B$ can expressed as
\[B=({\bf I}_{r-q+1}+C_\delta)^{1/2}-{\bf I}_{r-q+1}=\sum_{\ell=q}^r\mu_\ell\psi_\ell \psi_\ell^\top=:\Psi \widetilde{\Lambda} \Psi^{\top},\]
where $\mu_\ell:=\sqrt{1+\lambda_\ell}-1$, and $\widetilde{\Lambda}$ is a diagonal matrix with entries $\{\mu_\ell\}_{\ell=q}^r$.
Consequently, 
\begin{align*}
\mathbb{E} \sup_{{\bf v}\in \frac{1}{\delta_p}[-1,1]^{r-q+1} }|{\bf v}^\top B \mathcal{E}|=\frac{1}{\delta_p}\mathbb{E}\lVert B \mathcal{E}\rVert_1\le & \frac{\sqrt{p}}{\delta_p}\mathbb{E} \lVert B \mathcal{E}\rVert_2=\frac{\sqrt{p}}{\delta_p}\mathbb{E} \lVert \Psi \widetilde{\Lambda}  \mathcal{F}\rVert_2=\frac{\sqrt{p}}{\delta_p}\mathbb{E}\sqrt{\sum_{\ell=q}^r \mu_\ell^2 \mathcal{F}_\ell^2},
\end{align*}
where $\mathcal{F}_\ell:=\psi_\ell^{\top}\mathcal{E}$ are i.i.d. $\mathcal{N}(0,1)$ for $q\le \ell \le r$. By Jensen's inequality, the last expression is bounded by \[\frac{\sqrt{p}}{\delta_p}\sqrt{\sum_{\ell=q}^r \mu_\ell^2} \le C\frac{\sqrt{p}}{\delta_p}\sqrt{\sum_{\ell=q}^r\lambda_\ell^2}= C\frac{\sqrt{p}}{\delta_p} \sqrt{\sum_{i,j=1}^{p}  C_{\delta}(i,j)^2} \le C \frac{\sqrt{p}}{\delta_p^2} \sqrt{\sum_{i, j=1}^{p} A_p(i,j)^2}\]
where $C:=\sup_{x \geq -1 }\Big|\frac{\sqrt{1+x}-1}{x}\Big|<\infty$, and the last bound uses \eqref{eq:mean_field_A}. The last term above is $o(p)$ by the choice of $\delta_p$, and so we have verified \eqref{eq:vdelta}. This completes the proof of the Lemma.
\end{proof} 

To establish Theorem~\ref{thm:var_conv_new}, we need the following definitions. 
\begin{definition}
\label{def:prob_dist} 
Let $\xi_1, \cdots, \xi_p \sim \mathcal{N}(0,1)$ be iid random variables obtained from Lemma \ref{lem:off_diagonal}. 
Let $X \sim U([0,1])$ be independent of the $\xi_i$ variables. Given $\mathbf{u} \in [-1,1]^p$ set $Z = \xi_i$ and $U= u_i$  if $X \in [\frac{(i-1)}{p} , \frac{i}{p} ]$. Define ${\tilde{L}}_{ p}^{(\mathbf{u})}$ to be the joint distribution of $(X,Z,U)$. 
\end{definition}

Fix $W\in \mathcal{W}$, and $\phi,\psi$ are $L^1$ functions on $[0,1]$. Let $\tilde{\mathscr{F}}_{2, 4}$ denote the space of all joint distributions $(X,Z,U) \sim \nu$ such that $X \sim  U([0,1])$, $\mathbb{E}_{\nu}[Z^2 ] \leq  4$, $|U| \leq 1$. 
Define the functional $\tilde{\mathcal{G}}_{W, \phi, \psi} : \tilde{\mathscr{F}}_{2,4}  \to \mathbb{R} \cup \{ - \infty\}$ such that 
\begin{align}
\tilde{\mathcal{G}}_{W,\phi,\psi}(\nu) = \frac{1}{\sigma^2} \Big[ - \frac{1}{2} \mathbb{E}[W(X_1, X_2) U_1 U_2] + \mathbb{E}[\phi(X) U ] + \mathbb{E}[\sqrt{\psi(X)} U Z] \Big] - \mathbb{E}\Big[G\Big(U, \frac{\psi(X)}{\sigma^2} \Big) \Big], \label{eq:functional} 
\end{align}  
where  $(X_1, Z_1, U_1), (X_2, Z_2, U_2)$ are iid copies from $\nu$. 
Finally,  let $\tilde{\mathscr{F}}$ denote the space of all probability measures $\nu$ on $\R^3$, such that if $(X,Z,U)\sim  \nu$, then we have 
\begin{align}
 X \sim U([0,1]), Z \sim \mathcal{N}(0,1), X \perp \!\!\! \perp Z,  |U| \leq 1 \, \mathrm{a.s.}  \label{eq:domain} 
\end{align} 
We note that $\tilde{\mathscr{F}} \subset \tilde{\mathscr{F}}_{2,4}$, an observation that will be helpful in our subsequent analysis.

\noindent 
The following stability estimates will be crucial in our proof of Theorem~\ref{thm:var_conv_new}. The proof is deferred to the Appendix. 

\begin{lemma}\label{lem:rate2}
\begin{itemize}
\item[(i)]  We have, for $W, W' \in \mathcal{W}$, $\phi,\psi, \phi', \psi' \in L^1([0,1])$, 
\begin{align}
\sup_{\nu \in \tilde{\mathscr{F}}_{2,4}} |\tilde{\mathcal{G}}_{W, \phi , \psi}(\nu) - \tilde{\mathcal{G}}_{W', \phi', \psi'}(\nu)|  \lesssim  \| W- W'\|_{\square} + \| \phi - \phi' \|_1 + \| \psi - \psi'\|_1. \nonumber 
\end{align}

\item[(ii)]
Suppose  the following assumptions hold:

\begin{enumerate}
\item[(a)]
$W_k,W\in \mathcal{W}$ is such that $d_\square(W_k,W)\to 0$.

\item[(b)]
$\phi_k, \phi\in \mathcal{L}$ is such that $\int_{[0,1]}|\phi_k(x)-\phi(x)|dx \to 0$.

\item[(c)]
$\psi_k,\psi\in L_1[0,1]$ is such that $\int_{[0,1]}|\psi_k(x)-\psi(x)|dx\to 0$.

\end{enumerate}

Then we have
\begin{align}
\sup_{\nu \in \tilde{\mathscr{F}}_{2,4} } | \tilde{\mathcal{G}}_{W_k, W_k \cdot \phi_k + \phi_k \psi_k, \psi_k}(\nu) - \tilde{\mathcal{G}}_{W, W \cdot \phi + \phi \psi, \psi}(\nu)| \to 0 
\end{align}
as $k \to \infty$.  
\end{itemize} 
\end{lemma}

\begin{lemma}\label{lemma:sup_equiv} 
We have, for any $W \in \mathcal{W}$, $g, \psi \in L^1([0,1])$, 
\begin{align}
\sup_{ \nu \in \tilde{\mathscr{F}} } \tilde{\mathcal{G}}_{W,g,\psi} (\nu) = \sup_{F \in \mathscr{F}}  \mathcal{G}_{W,g,\psi} (F). \nonumber 
\end{align} 
Further, both the suprema are attained. 
\end{lemma}

\begin{remark} 
\label{rem:case_analysis} 
Depending on whether $D_{\mathrm{KL}}(\pi_{\infty} \| \pi)$ and $D_{\mathrm{KL}}(\pi_{-\infty} \| \pi)$ are infinity or not, there are four possible cases. In our subsequent proofs, we consider the case $D_{\mathrm{KL}}(\pi_{\infty} \| \pi) < \infty$ and $D_{\mathrm{KL}}(\pi_{-\infty} \| \pi) = \infty$, noting that other cases follow by natural modifications. 
\end{remark} 

\begin{lemma}\label{lemma:tightness} 
Let $d(\cdot, \cdot)$ be the $2$-Wasserstein distance on $\mathrm{Pr}([0,1] \times \mathbb{R} \times [-1,1])$. For any $\delta>0$, set 
\begin{align}
\mathscr{B}_{d}( \tilde{\mathscr{F}}, \delta) = \{ \nu \in \mathrm{Pr}([0,1] \times \mathbb{R} \times [-1,1])  : \inf_{\nu' \in \tilde{\mathscr{F}}} d ( \nu , \nu') < \delta\}. \nonumber 
\end{align} 
Then there exists a sequence $\delta_p \to 0$ as $p \to \infty$ such that setting
\begin{align}
F_p = \{  \forall \mathbf{u} \in [-1,1]^p , \,\,\, \tilde{L}_p^{(\mathbf{u})}  \in \mathscr{B}_{d}( \tilde{\mathscr{F}}, \delta_p )   \}  \nonumber 
\end{align} 
we have, $\mathbb{P} ( F_p | \mathbf{X} ) = 1- o(1)$.
\end{lemma} 

\noindent
The proof of Lemma \ref{lemma:tightness} is straightforward, and is thus omitted.  

\begin{lemma}\label{lemma:semicontinuity} 
Let $\{\mathbf{u}_p : p \geq 1\} \in [-1,1]^{p}$ be such that $\tilde{L}_{p}^{(\mathbf{u}_p)}$ converges weakly to $\nu_0 \in \tilde{\mathscr{F}}$. Then for any $W \in \mathcal{W}$, $g, \psi \in L^1$,
\begin{align}
\limsup_{p \to \infty} \tilde{\mathcal{G}}_{W,g,\psi}(\tilde{L}_p^{(\mathbf{u}_{p} )}) \leq \tilde{\mathcal{G}}_{W,g,\psi}(\nu_0). \nonumber 
\end{align} 
\end{lemma} 

\begin{lemma}\label{lemma:lower_bound_argument} 
Suppose we are in the setting of Theorem \ref{thm:var_conv_new}. 
Fix $F \in \mathscr{F}$ and $\delta, K >0$ and $p \geq 1$. Let $V_i \sim U([(i-1)/p, i/p])$ be independent of $\xi_i$ arising in Lemma \ref{lem:off_diagonal}. Set $d_i^{K} = d_i \mathbf{1}(|d_i| \leq K)$ and define 
\begin{align} 
\tilde{u}_i = \begin{cases} 
F(V_i, \xi_i) & \textrm{if} \,\,\,  F(V_i, \xi_i)  \geq -1+ \delta,  \\
\dot{c}(0, d_i^{K}) & \textrm{o.w.} 
\end{cases} \label{eq:u_defn} 
\end{align}
Then we have, for any $\varepsilon>0$
\begin{align}
\limsup_{K \to \infty} \limsup_{\delta \to 0} \limsup_{p \to \infty} \mathbb{P}\Big[  \Big| \frac{1}{p} \widetilde{M}_p(\mathbf{u}) - \mathcal{G}_{W,g,\psi}(F)  \Big| > \varepsilon  | \mathbf{X}\Big] =0. \nonumber  
\end{align} 
\end{lemma} 

\noindent
We prove Theorem~\ref{thm:var_conv_new} assuming Lemma \ref{lemma:sup_equiv}, \ref{lemma:semicontinuity} and \ref{lemma:lower_bound_argument}. The corresponding proofs are deferred to the end of the section. 

\begin{proof}[Proof of Theorem~\ref{thm:var_conv_new}]
Using Lemma~\ref{lem:off_diagonal}, it suffices to prove that 
\begin{align}
\frac{1}{p} \sup_{\mathbf{u} \in [-1,1]^p} \widetilde{M}_p({\bf u}) \stackrel{P|\mathbf{X}}{\to}  \sup_{F \in \mathscr{F}}  \mathcal{G}_{W,g,\psi} (F). \nonumber 
\end{align} 

\vspace{2pt}
\noindent
\textbf{Upper bound:} 
Note that 
\begin{align}
\frac{1}{p}  \widetilde{M}_p({\bf u})\stackrel{(b)}{=} & \tilde{\mathcal{G}}_{W_{p A_p},  W_{p A_p} \cdot w_{\bbeta_0} + w_{\bbeta_0} w_{\mathbf{D}}, w_{\sigma^2 \mathbf{d}} } (\tilde{L}_p^{(\mathbf{u})} ) \nonumber \\ 
\stackrel{(c)}{=} &  \tilde{\mathcal{G}}_{W,  g, \psi } (\tilde{L}_p^{(\mathbf{u})} )  +  \mathcal{E}_p^{(2)} ({\bf u}), \nonumber
\end{align} 
where $\sup_{\mathbf{u} \in [-1,1]^p}   |\mathcal{E}_p^{(2)} ({\bf u})|   \stackrel{P|\mathbf{X}}{\to} 0$ as $p \to \infty$. Here, $(b)$ uses the definition \eqref{eq:functional}, and $(c)$ uses Lemma~\ref{lem:rate2} part (ii). To invoke Lemma \ref{lem:rate2}, we use the fact that 
\begin{align}
\mathbb{P}[ \forall \mathbf{u} \in [-1,1]^p, \tilde{L}_p^{(\mathbf{u})} \in \tilde{\mathscr{F}}_{2,4} ] = \mathbb{P} \Big[ \sum_{i=1}^{p}  \xi_i^2 \leq 4 p \Big] \to 1 \label{eq:event_equiv} 
\end{align} 
as $p \to \infty$.


To complete the upper bound, invoking Lemma \ref{lemma:sup_equiv}, it suffices to establish that for all $\varepsilon>0$ 
\begin{align}
\mathbb{P}\Big[ \sup_{\mathbf{u} \in [-1,1]^p} \tilde{\mathcal{G}}_{W,g,\psi}(\tilde{L}_p^{(\mathbf{u})}) <  \sup_{\nu \in \tilde{\mathscr{F}}} \tilde{\mathcal{G}}_{W,g,\psi}(\nu) + \varepsilon | \mathbf{X} \Big] = 1 - o(1) . \label{eq:thm2_conclusion} 
\end{align} 
We first turn to the proof of \eqref{eq:thm2_conclusion} by contradiction. Suppose there exists  $\varepsilon>0$, such that setting  
\begin{align}
E_p  =  \Big\{  \sup_{\mathbf{u} \in [-1,1]^p} \tilde{\mathcal{G}}_{W,g,\psi}(\tilde{L}_p^{(\mathbf{u})}) > \sup_{\nu \in \tilde{\mathscr{F}}} \tilde{\mathcal{G}}_{W,g,\psi}(\nu)  + \varepsilon \Big\}  . \nonumber 
\end{align} 
we have $\limsup \, \mathbb{P}[ E_p | \mathbf{X} ] >0$. 
Using Lemma \ref{lemma:tightness}, we have, $\limsup \mathbb{P}[ F_p \cap E_p ] >0$. Therefore, there exists a subsequence $p_k \to \infty$ along which $E_{p_k} \cap F_{p_k} \neq \emptyset$. 

Consequently, there exists $\bm{\xi}_{p_k} := (\xi_{1, p_k} , \cdots, \xi_{p_k, p_k}) \in \mathbb{R}^{ p_k}$, $\bm{\xi}_{p_k} \in E_{p_k} \cap F_{p_k}$ and $\mathbf{u}_{p_k} \in [-1,1]^{p_k}$ such that 
\begin{align}
\tilde{\mathcal{G}}_{W,g,\psi}(\tilde{L}_p^{(\mathbf{u}_{p_k} )}) > \sup_{\nu \in \tilde{\mathscr{F}}} \tilde{\mathcal{G}}_{W,g,\psi}(\nu)  + \varepsilon, \,\,\,\,\, \tilde{L}_{p_k}^{(\mathbf{u}_{p_k})}  \in \mathscr{B}_{d}( \tilde{\mathscr{F}}, \delta_p ) . \nonumber
\end{align}


This in turn implies that the sequence $\{ \tilde{L}_{p_k}^{(\mathbf{u}_{p_k})} : k \geq 1\} $  is tight, and hence has a subsequence converging weakly to $\nu_0 \in \tilde{\mathscr{F}}$. Assuming without loss of generality that $\tilde{L}_{p_k}^{(\mathbf{u}_{p_k})} \stackrel{d}{\to} \nu_0$, the desired contradiction follows once we show that 
\begin{align}
\limsup_{k \to \infty} \tilde{\mathcal{G}}_{W,g,\psi}(\tilde{L}_p^{(\mathbf{u}_{p_k} )}) \leq \tilde{\mathcal{G}}_{W,g,\psi}(\nu_0). \label{eq:contradiction1} 
\end{align} 

But this follows from Lemma \ref{lemma:semicontinuity}.

\vspace{5pt} 
\noindent
\textbf{Lower Bound:} 
 It suffices to show that for all $\varepsilon >0$, 
\begin{align}
\mathbb{P}\Big[ \frac{1}{p} \sup_{\mathbf{u} \in [-1,1]^p} \widetilde{M}_p({\bf u}) > \sup_{F \in \mathscr{F}}  \mathcal{G}_{W,g,\psi} (F) - \varepsilon | \mathbf{X} \Big] = 1 -o(1). \nonumber 
\end{align} 
To this end, let $F \in \mathscr{F}$ satisfy 
\begin{align}
\mathcal{G}_{W,g,\psi} (F) > \sup_{F \in \mathscr{F}}  \mathcal{G}_{W,g,\psi} (F) - \frac{\varepsilon}{2}. \nonumber 
\end{align} 
Note that if $\mathbb{E}[G(F(X, Z), \psi(X)] = \infty$, then the lower bound is trivial. Thus we assume, without loss of generality, that $\mathbb{E}[G(F(X, Z), \psi(X)] < \infty$. Next, fixing $K>0$, we have, 
\begin{align}
\mathcal{G}_{W,g,\psi_K} (F) = \mathcal{G}_{W,g,\psi} (F) + o_K(1), \label{eq:psi_trunc} 
\end{align} 
where $\psi_K(x) := \psi(x) \mathbf{1}(\psi(x) \leq K)$, and we have used   Lemma \ref{lem:rate2} Part (i). Further, setting $d_i^{K} = \min\{ d_i , K \}$, a similar argument gives
\begin{align}
& \tilde{\mathcal{G}}_{W_{p A_p},  W_{p A_p} \cdot w_{\bbeta_0} + w_{\bbeta_0} w_{\mathbf{D}}, w_{\sigma^2 \mathbf{d}} } (\tilde{L}_p^{(\mathbf{u})} ) \nonumber \\
 &=  \tilde{\mathcal{G}}_{W_{p A_p},  W_{p A_p} \cdot w_{\bbeta_0} + w_{\bbeta_0} w_{\mathbf{D}}, w_{\sigma^2 \mathbf{d}^{K}} } (\tilde{L}_p^{(\mathbf{u})} ) + o_K(1). \label{eq:tildeG-trunc}
\end{align}
  For any $p \geq 1$, let $\xi_1, \cdots, \xi_p$ denote the iid $\mathcal{N}(0,1)$ random variables arising in the optimization problem $M_p$. For each $i \in [p]$, generate $V_i \sim U([(i-1)/p, i/p])$ independent of each other, and of the $\xi_i$ variables. 
Fixing $\delta>0$, we set 
\begin{align} 
\tilde{u}_i = \begin{cases} 
F(V_i, \xi_i) & \textrm{if} \,\,\,  F(V_i, \xi_i)  \geq -1+ \delta,  \\
\dot{c}(0, d_i^{K}) & \textrm{o.w.} 
\end{cases}  \label{eq:u-tilde_defn} 
\end{align} 

\noindent 
Using Lemma \ref{lemma:lower_bound_argument}, for any $\varepsilon>0$, there exists $\delta, K>0$ (depending on $\varepsilon$), such that with probability at least $1- \varepsilon$, 
\begin{align}
\frac{1}{p} \sup_{\mathbf{u} \in [-1,1]^p} \widetilde{M}_p({\bf u}) \geq \frac{1}{p} \widetilde{M}_p(\tilde{\mathbf{u}}) \geq \mathcal{G}_{W,g,\psi} (F) -  \varepsilon \geq   \sup_{F \in \mathscr{F}}  \mathcal{G}_{W,g,\psi} (F) - 2 \varepsilon. \nonumber 
\end{align}  

\noindent
As $\varepsilon>0$ is arbitrary, this completes the proof of the lower bound. 
\end{proof} 

\noindent
It remains to prove Lemma \ref{lemma:sup_equiv}, \ref{lemma:semicontinuity} and \ref{lemma:lower_bound_argument}. We establish each result in turn. 

\begin{proof}[Proof of Lemma \ref{lemma:sup_equiv}]
To begin, note that for $F \in \mathscr{F}$, $(X,Z, F(X,Z)) \sim \nu  \in \tilde{\mathscr{F}}$. Therefore, 
\begin{align} 
\sup_{ \nu \in \tilde{\mathscr{F}} } \tilde{\mathcal{G}}_{W,g,\psi} (\nu) \geq  \sup_{F \in \mathscr{F}}  \mathcal{G}_{W,g,\psi} (F). \nonumber 
\end{align} 
We now turn to the reverse inequality.  Fix $(X,Z,U) \sim \nu \in \tilde{\mathscr{F}}$. 
Setting $V:= \Phi(Z)$ note that $X,V$ are $U([0,1])$ iid. Define $F(X,V) = \mathbb{E}[U|X,V]$. Note that 
\begin{align}
\mathbb{E}[W(X, X') U U' ] &= \mathbb{E}[W(X, X') F(X, V) F(X', V')] \nonumber \\
 \mathbb{E}[g(X) U ] &= \mathbb{E}[g(X) F(X,V)] ,\nonumber \\
 \mathbb{E} \Big[G\Big(U, \frac{\psi(X)}{\sigma^2} \Big) \Big] &\geq \mathbb{E} \Big[G \Big(F(X,V), \frac{\psi(X)}{\sigma^2} \Big) \Big], \nonumber 
\end{align}  
where the final step follows from Jensen's inequality and the observation that $\frac{\partial^2}{\partial u^2} G(u,d) \geq 0$ from Lemma \ref{lemma:G_stability}. This implies 
\begin{align}
\tilde{\mathcal{G}}_{W,g,\psi} (\nu) 
&\leq \frac{1}{\sigma^2} \Big[ - \frac{1}{2} \mathbb{E}[W(X, X') F(X,V) F(X,V')] + \mathbb{E}[ g(X) F(X,V) ] + \mathbb{E}[\sqrt{\psi(X)} F(X,V) \Phi^{-1}(V)] \Big] \nonumber \\
&- \mathbb{E}\Big[G \Big(F(X,V), \frac{\psi(X)}{\sigma^2} \Big) \Big]. \nonumber 
\end{align} 
Finally, noting that $\Phi^{-1}$ is strictly increasing, we have, 
\begin{align}
\mathbb{E}[\sqrt{\psi(X)} F(X,V) \Phi^{-1}(V)] \leq \mathbb{E}[\sqrt{\psi(X)} F^*(X,V) \Phi^{-1}(V)] , \nonumber 
\end{align} 
where for each $x \in [0,1]$, $F^*(x,\cdot)$ is the monotone rearrangment of $F(x, \cdot)$. This last step follows from Hardy-Littlewood inequality. Observe that as the Uniform distribution is invariant under measure preserving transformations, the other terms in the functional above are unchanged by this rearrangement. Thus 
\begin{align}
 \tilde{\mathcal{G}}_{W,g,\psi} (\nu) \leq \mathcal{G}_{W,g,\psi} (F^*). \nonumber 
\end{align} 
Since $F^* \in \mathscr{F}$, the proof is complete. 

Finally, we show that both suprema are attained. To this end, observe that by Lemma \ref{lemma:elementary} Part (iv),  
 $\tilde{\mathcal{G}}_{W,g,\psi}(\cdot)$ is upper semi-continuous. Moreover, $\tilde{\mathscr{F}}$ is compact under weak topology---this follows as the collection $\tilde{\mathscr{F}}$ is tight, and thus pre-compact by Prokhorov's theorem. Thus a maximum exists. To see that $\mathcal{G}_{W,\phi,\psi}(\cdot)$ attains it's maximum, start with a maximizer $\nu_0$ of $\tilde{\mathcal{G}}_{W,\phi,\psi}(\cdot)$, and take $F_0^*$ as obtained in the proof above. 
\end{proof} 

\begin{proof}[Proof of Lemma \ref{lemma:semicontinuity}]
Observe that we can approximate $W$ in $L^1$ by a  continuous function $W^{(\ell)}$. 
This implies 
\begin{align}
\mathbb{E}_{\tilde{L}_p^{(\mathbf{u})}} [W(X_1,X_2) U_1 U_2] &= \mathbb{E}_{\tilde{L}_p^{(\mathbf{u})}} [W^{(\ell)}(X_1,X_2) U_1 U_2]  + o_\ell(1) \nonumber \\
&\stackrel{p \to \infty}{\to} \mathbb{E}_{\nu_0} [W^{(\ell)}(X_1,X_2) U_1 U_2] + o_{\ell}(1) = \mathbb{E}_{\nu_0} [W(X_1,X_2) U_1 U_2] + o_{\ell}(1). \nonumber 
\end{align} 
Upon sending $\ell \to \infty$, we conclude that $\mathbb{E}_{\tilde{L}_p^{(\mathbf{u})}} [W(X_1,X_2) U_1 U_2] \to  \mathbb{E}_{\nu_0} [W(X_1,X_2) U_1 U_2]$. A similar argument shows that 
%
\begin{align}
\mathbb{E}_{\tilde{L}_p^{(\mathbf{u})}}[g(X) U] \stackrel{p \to \infty}{\to} \mathbb{E}_{\nu_0}[g(X)  U].\nonumber 
\end{align} 
Next, we note that we can approximate $\psi$ in $L^1$ by a sequence of  continuous functions $\{\psi^{(\ell)}: \ell \geq 1\}$, and thus, by Cauchy-Schwarz
\begin{align}
| \mathbb{E}_{\tilde{L}_p^{(\mathbf{u})}} [ \sqrt{\psi(X)} U Z] - \mathbb{E}_{\tilde{L}_p^{(\mathbf{u})}} [ \sqrt{\psi^{(\ell)}(X)} U Z] | &\leq \mathbb{E}_{\tilde{L}_p^{(\mathbf{u})}} [|\sqrt{\psi}^{(\ell)}(X) - \sqrt{\psi(X)}|\cdot|Z|] \nonumber \\
&\leq \sqrt{\mathbb{E}_{\tilde{L}_p^{(\mathbf{u})}} [(\sqrt{\psi^{(\ell)}(X)} - \sqrt{\psi(X)})^2] \mathbb{E}_{\tilde{L}_p^{(\mathbf{u})}} [ Z^2]}  \nonumber\\
&\leq \sqrt{2 \int_{0}^{1} | \psi^{(\ell)} (x) - \psi(x) | \mathrm{d}x} = o_{\ell}(1). \nonumber 
\end{align} 
The last inequality uses that $\tilde{L}_{p}^{(\mathbf{u}_{p}) } \in \tilde{\mathscr{F}}_{2,4}$. 
 This implies 
\begin{align}
\mathbb{E}_{\tilde{L}_p^{(\mathbf{u})}} [ \sqrt{\psi(X)} U Z] &= \mathbb{E}_{\tilde{L}_p^{(\mathbf{u})}} [ \sqrt{\psi^{(\ell)}(X)} U Z] + o_{\ell}(1) \nonumber \\
&\stackrel{p \to \infty}{\to} \mathbb{E}_{\nu_0} [ \sqrt{\psi^{(\ell)}(X)} U Z] + o_{\ell}(1)  \nonumber \\
&= \mathbb{E}_{\nu_0} [ \sqrt{\psi(X)} U Z]  + o_{\ell}(1).  \nonumber 
\end{align} 
Sending $\ell \to \infty$, we have $\mathbb{E}_{\tilde{L}_p^{(\mathbf{u})}} [ \sqrt{\psi(X)} U Z] \to \mathbb{E}_{\nu_0} [ \sqrt{\psi(X)} U Z]$ as $p \to \infty$.  

The desired conclusion follows once we establish that 
\begin{align}
\liminf_{p \to \infty} \mathbb{E}_{\tilde{L}_p^{(\mathbf{u})}}\Big [G \Big(U, \frac{\psi(X)}{\sigma^2} \Big)  \Big] \geq \mathbb{E}_{\nu_0} \Big[G \Big(U, \frac{\psi(X)}{\sigma^2}  \Big) \Big]. \nonumber 
\end{align}
%
%
But this follows on using Fatou's lemma and the lower semicontinuity of $G \Big(\cdot, \frac{ \psi(X)}{\sigma^2} \Big)$  (Lemma \ref{lemma:elementary} part (iv)). 
\end{proof}

\begin{proof}[Proof of Lemma \ref{lemma:lower_bound_argument}] 
The definition of $\mathbf{u}$ in \eqref{eq:u_defn} implies 
\begin{align}
G(u_i, d_i^{K} ) = \begin{cases}
G(F(V_i, \xi_i), d_i^{K}) & \textrm{if} \,\,\, F(V_i, \xi_i) \geq  -1 + \delta, \\
0 &\textrm{o.w.} 
\end{cases} \nonumber 
\end{align} 
Observe that 
\begin{align}
\frac{1}{p} \sum_{i=1}^{p} G(u_i,  d_i^{K} ) = &\frac{1}{p} \sum_{i=1}^{p} G( F(V_i, \xi_i) , d_i^K) \mathbf{1}(  F(V_i, \xi_i) \geq -1 + \delta)  \label{eq:var_exp4} 
\end{align} 
Now, the function $G : [-1+ \delta,1] \times [0, K] \to \mathbb{R}^+$ is bounded, 
and thus by Chebychev inequality, 
\begin{align}
&\frac{1}{p} \sum_{i=1}^{p} G( F(V_i, \xi_i), d_i^K) \mathbf{1}(  F(V_i, \xi_i)  \geq -1 + \delta) \nonumber \\
-& \frac{1}{p} \sum_{i=1}^{p} \mathbb{E}\Big[ G( F(V_i, \xi_i) , d_i^{K}) \mathbf{1}( F(V_i, \xi_i) \geq -1 + \delta) \Big] \stackrel{P|\mathbf{X}}{\to} 0 . \label{eq:var_exp1} 
\end{align} 
Also, note that 
\begin{align}
& \frac{1}{p} \sum_{i=1}^{p} \mathbb{E}\Big[ G( F(V_i, \xi_i) , d_i^K) \mathbf{1}(  F(V_i, \xi_i)  \geq -1 + \delta) \Big] \nonumber \\
  &= \mathbb{E}[G(F(X,Z), w_{\mathbf{d}^K,p}(X)) \mathbf{1}(  F(X, Z)  \geq -1 + \delta)] , \nonumber \\
  &= \mathbb{E}\Big[G\Big(F(X,Z), \frac{\psi_K(X) }{\sigma^2} \Big) \mathbf{1}(  F(X, Z)  \geq -1 + \delta) \Big]  +o(1)  \label{eq:var_exp3} 
\end{align} 
where $X \sim U([0,1])$ and $Z \sim \mathcal{N}(0,1)$ are independent. The last display uses Lemma \ref{lemma:G_stability} along with the fact that $w_{\mathbf{d}^K,p} \to \psi_K/\sigma^2$ in $L^1$ for almost every $K>0$. 
Now, the assumptions $\mathbb{E}[G(F(X, Z), \psi(X)/\sigma^2)] < \infty$ and $D_{\mathrm{KL}}(\pi_{-\infty} \| \pi) = \infty$ together imply $ \mathbb{P}[F(X,Z) = -1] = 0$. In combination with Dominated Convergence Theorem, this implies 
\begin{align}
\lim_{\delta \to 0} \mathbb{E} \Big[G \Big(F(X,Z), \frac{\psi_K(X)}{\sigma^2}  \Big) \mathbf{1}(  F(X, Z)  \geq -1 + \delta) \Big] =  \mathbb{E}\Big[G \Big(F(X,Z),  \frac{\psi_K(X)}{\sigma^2}  \Big) \Big]. \label{eq:var_exp2} 
\end{align} 
Combining \eqref{eq:var_exp4}, \eqref{eq:var_exp1}, \eqref{eq:var_exp3} and \eqref{eq:var_exp2}, for any $\varepsilon>0$, we have, 
\begin{align}
\limsup_{\delta \to 0} \limsup_{p \to \infty} \mathbb{P} \Big[ \Big| \frac{1}{p} \sum_{i=1}^{p} G(u_i,  d_i^{K} )  - \mathbb{E} \Big[G\Big(F(X,Z), \frac{\psi_K(X)}{\sigma^2} \Big) \Big ] \Big| > \varepsilon | \mathbf{X} \Big] =0. \label{eq:G_lim} 
\end{align} 
%
%
%
We now claim that for any $\varepsilon>0$, 
\begin{align}
\begin{split} \label{eq:var_otherterms} 
\limsup_{\delta \to 0} \limsup_{p \to \infty} \mathbb{P}\Big[ \Big| \frac{1}{p} \sum_{i,j=1}^{p} A_p(i,j) u_i u_j  - \mathbb{E}[W(X_1, X_2) F(X_1, Z_1) F(X_2, Z_2)] \Big|  > \varepsilon | \mathbf{X} \Big] = 0.  \\
\limsup_{\delta \to 0} \limsup_{p \to \infty} \mathbb{P}\Big[ \Big| {\bf u}^{\mathrm{T}} A_p \bbeta_0+{\bf u}^{\mathrm{T}} D_p \bbeta_0 - \mathbb{E}[g(X) F(X,Z)] \Big|  > \varepsilon | \mathbf{X} \Big] = 0.  \\
\limsup_{\delta \to 0} \limsup_{p \to \infty} \mathbb{P}\Big[ \Big| \sum_{i=1}^p u_i \sqrt{D_p(i,i)} \xi_i - \mathbb{E}[\sqrt{\psi(X)} UZ] \Big|  > \varepsilon | \mathbf{X} \Big] = 0. 
\end{split} 
\end{align} 

We first turn to the quadratic form. We have, 
\begin{align}
\frac{1}{p} \sum_{ij=1}^{p} A_p(i,j) u_i u_j & = \mathbb{E}_{\tilde{L}_p^{(\mathbf{u})}} [W_{pA_p}(X_1, X_2) U_1 U_2]  \nonumber \\
&\stackrel{(a)}{=} \mathbb{E}_{\tilde{L}_p^{(\mathbf{u})}} [W(X_1, X_2) U_1 U_2]  +o(1) \nonumber  \\
&\stackrel{(b)}{=} \mathbb{E}_{\tilde{L}_p^{(\mathbf{u})}} [W^{(K)}(X_1, X_2) U_1 U_2]  + o_K(1). \label{eq:quad_1} 
\end{align} 
Here, the step (a) uses Assumption (a), and step (b) uses $W^{(K)} := W \mathbf{1}(|W| \leq K) \stackrel{L^1}{\to} W$. We have, setting $\mathcal{M}_p(i,j) := \int_{[(i-1)/p, i/p]} \int_{[(j-1)/p, j/p]} W^{(K)}(x,y) \mathrm{d}x \mathrm{d}y$, 
\begin{align}
\mathbb{E}_{\tilde{L}_p^{(\mathbf{u})}} [W^{(K)}(X_1, X_2) U_1 U_2] &= \sum_{i,j=1}^{p} \mathcal{M}_{p}(i,j) u_i u_j  = \sum_{i,j=1}^{p} \mathcal{M}_{p}(i,j) \mathbb{E}[u_i] \mathbb{E}[ u_j] + o_P(1) \label{eq:quad_2} 
\end{align} 
by the weak law of large numbers, and the observation $| \mathcal{M}_p(i,j)| \leq K/p^2$. Here, the $o_P(1)$ term converges to zero in probability under the joint distribution of $\{(V_i,\xi_i): 1 \leq i \leq p\}$.  Further, we have,  
\begin{align}
\Big|\sum_{i,j=1}^{p} \mathcal{M}_{p}(i,j) \mathbb{E}[u_i] \mathbb{E}[ u_j] - \sum_{i,j=1}^{p} \mathcal{M}_p(i,j) \mathbb{E}[F(V_i, Z)] \mathbb{E}[F(V_j,Z)] \Big| &\leq    2 K  \mathbb{P}[ F(X,Z) \leq -1 +\delta], \label{eq:quad_3} 
\end{align} 
where $X \sim U([0,1])$ and $Z \sim \mathcal{N}(0,1)$ are independent. Finally, 
\begin{align}
\sum_{i,j=1}^{p} \mathcal{M}_p(i,j) \mathbb{E}[F(V_i, Z)] \mathbb{E}[F(V_j,Z)]  &= \mathbb{E}[W_{\mathcal{M}_p}(X,X') F(X, Z) F(X', Z')] \nonumber \\
&\stackrel{(a)}{=} \mathbb{E}[W^{(K)}(X,X') F(X, Z) F(X', Z')] + o(1) \nonumber \\
&\stackrel{(b)}{=} \mathbb{E}[W(X,X') F(X, Z) F(X', Z')] + o_K(1) \label{eq:quad_4} 
\end{align} 
where (a) uses $\|W_{\mathcal{M}_p} - W^{(K)} \|_1 \to 0$ as $p \to \infty$ and (b) uses $\|W - W^{(K)} \|_1 \to 0$ as $K \to \infty$. Combining \eqref{eq:quad_1}, \eqref{eq:quad_2}, \eqref{eq:quad_3}, \eqref{eq:quad_4}, the first conclusion of \eqref{eq:var_otherterms} follows upon sending $p \to \infty$, $\delta \to 0$ and $K \to \infty$. 

The other conclusions of \eqref{eq:var_otherterms} follow using analogous arguments, and are thus omitted. 
\end{proof} 

\subsection{Proof of Theorem~\ref{thm:pure_state} and related results} 
\label{subsec:proof_thm3}
The following continuity statement will be critical for the proof of Theorem \ref{thm:pure_state}. The proof is deferred to the Appendix. 
\begin{lemma}
\label{lem:map_cont} 
Define a map $m : \mathrm{Pr}\Big([0,1] \times \mathbb{R} \times [-1,1] \Big) \times \mathcal{W} \times [0,1] \to \mathbb{R}$, $(\mu , W, x) \mapsto \mathbb{E}_{(X,Z,U) \sim \mu} [W(x,X') U']$. 
\begin{itemize}
\item[(i)] Fix $x \in [0,1]$, $\mu \in \mathrm{Pr}\Big([0,1] \times \mathbb{R} \times [-1,1] \Big)$, and $W, W' \in \mathcal{W}$. For any $\varepsilon>0$, set $S(\varepsilon)= \{ x : |m(\mu, W, x) - m(\mu, W', x) | > \varepsilon\}$. For $X \sim U([0,1])$, we have, 
\begin{align}
\mathbb{P}[X \in S(\varepsilon)] \leq \frac{2}{\varepsilon} \|W - W' \|_{\square}. \nonumber 
\end{align} 


\item[(ii)] For any $W \in \mathcal{W}$, 
if $\nu_p \stackrel{w}{\to} \nu$, then 
\begin{align}
\sup_{x \in [0,1]} \Big| m(\nu_p, W, x) -  m(\nu, W, x)  \Big| \to 0 \nonumber 
\end{align} 
as $p \to \infty$. 
\end{itemize} 
\end{lemma} 

\begin{proof}[Proof of Theorem~\ref{thm:pure_state}] 
We break the proof into several steps. Note that using Lemma \ref{lemma:sup_equiv}, both variational problems attain their suprema.  

\begin{itemize}

\item[(i)]
We establish uniqueness of the maximizer assuming \eqref{eq:separate}. 
%
If possible, let $F_1$ and $F_2$ be two distinct optimizers of $\mathcal{G}_{W,g,\psi}$ in $\mathscr{F}$. Therefore, using Theorem \ref{thm:var_conv_new}, 
\begin{align}
\frac{1}{p}\sup_{{\bf u}\in [-1,1]^p}M_p({\bf u}) \stackrel{P|\mathbf{X}}{\longrightarrow} \sup_{F\in \mathscr{F}} \mathcal{G}_{W,g,\psi}(F)  = \mathcal{G}_{W,g,\psi}(F_1) = \mathcal{G}_{W,g,\psi}(F_2). \nonumber  
\end{align} 

Let $\xi_1, \cdots, \xi_p$ be iid $\mathcal{N}(0,1)$ random variables arising in the functional $\widetilde{M}_p(\cdot)$. Further, let $V_i \sim U([(i-1)/p, i/p])$ be independent of the $\xi_i$ variables. Fixing $\delta>0$, following the recipe in \eqref{eq:u_defn}, we define two sequences $\mathbf{u}^{(1)}_{p,\delta,K}$, $\mathbf{u}^{(2)}_{p,\delta,K}$ in $[-1,1]^p$ corresponding to $F_1, F_2$ respectively.  Using Lemma \ref{lemma:lower_bound_argument}, for $i=1,2$, for $\delta, K>0$ we have

\begin{align}
&\mathbb{P}\Big[  \frac{1}{p}M_p(\mathbf{u}^{(i)}_{p,\delta,K}) >  \mathcal{G}_{W,g,\psi}(F_i) - \frac{\delta'}{4}   \Big| \mathbf{X}\Big]  \geq  1 -  \mathcal{E}(\delta,K),  \nonumber 
\end{align} 
where $\limsup_{K \to \infty} \limsup_{\delta \to 0} \mathcal{E}(\delta,K) =0$. 

Moreover, using Theorem \ref{thm:var_conv_new} and the assumption that $F_1$ and $F_2$ are optimizers of the limiting variational problem, we have, for $i=1,2$, 
\begin{align}
\mathbb{P}\Big[  \frac{1}{p}M_p(\mathbf{u}^{*}_p) <  \mathcal{G}_{W,g,\psi}(F_i) + \frac{\delta'}{4} \Big| \mathbf{X} \Big] = 1- o(1). \nonumber 
\end{align} 

On the event $\{ \| \mathbf{u}_{p,\delta,K}^{(i)} - \mathbf{u}_p^* \|_2^2 > p \varepsilon , i=1,2\} \cap \{  \frac{1}{p}M_p(\mathbf{u}^{*}_p) <  \mathcal{G}_{W,g,\psi}(F_i) + \frac{\delta'}{4}\} $, for $i = 1,2$ we have,
\begin{align}
\frac{1}{p} M_p(\mathbf{u}_{p,\delta,K}^{(i)} ) \leq   \frac{1}{p} M_p(\mathbf{u}^{*}_p) - \delta' \leq \mathcal{G}_{W,g,\psi}(F_i) + \frac{\delta'}{4} - \delta' < \mathcal{G}_{W,g,\psi}(F_i) - \frac{\delta'}{4}. \nonumber 
\end{align} 

Using \eqref{eq:separate}, this implies, for all $p$ sufficiently large, $\mathbb{P}[ \| \mathbf{u}_{p,\delta,K}^{(i)} - \mathbf{u}_p^* \|_2^2 > p \varepsilon, i=1,2 | \mathbf{X} ] \leq 3 \mathcal{E}(\delta,K) $. By triangle inequality,  $\mathbb{P}[\| \mathbf{u}_{p,\delta,K}^{(1)} - \mathbf{u}_{p,\delta,K}^{(2)} \|_2^2 > 2 p \varepsilon | \mathbf{X}] \leq 3 \mathcal{E}(\delta,K) $. As $\mathbf{u}_{p,\delta,K}^{(i)} \in [-1, 1]^{p}$, we have, 
\begin{align}
\frac{1}{p}\mathbb{E}[ \| \mathbf{u}_{p,\delta,K}^{(1)} - \mathbf{u}_{p,\delta,K}^{(2)}\|_2^2 ] \leq 2 \varepsilon + 12 \mathcal{E}(\delta,K) . \nonumber 
\end{align} 
Thus, with $X \sim U([0,1])$, $Z \sim \mathcal{N}(0,1)$ independent, we have, 
\begin{align}
&\mathbb{E}[(F_1(X,Z) - F_2(X,Z))^2] \nonumber \\
=& \frac{1}{p} \sum_{i=1}^{p} \mathbb{E}[(F_1(V_i , \xi_i) - F_2 (V_i, \xi_i))^2] \nonumber \\
\leq&  \frac{1}{p}\mathbb{E}[ \| \mathbf{u}_{p,\delta,K}^{(1)} - \mathbf{u}_{p,\delta,K}^{(2)}\|_2^2 ]  + \sum_{a=1,2} \mathbb{P}[F_a(X,Z) \leq -1 + \delta] \nonumber \\
\leq& 2  \varepsilon + 12 \mathcal{E}(\delta,K) + \sum_{a=1,2} \mathbb{P}[F_a(X,Z) \leq -1 + \delta] . \nonumber 
\end{align} 

 Setting $\delta \to 0$ followed by $K \to \infty$, we obtain that $\mathbb{E}[(F_1(X,Z) - F_2(X,Z))^2] \leq 2 \varepsilon$.  Here we use that $F_1,F_2$ are optimizers; thus without loss of generality, $\mathbb{E}[G(F_a(X, Z), \psi(X)] < \infty$ for $a=1,2$, which along with Remark \ref{rem:case_analysis} implies that $ \mathbb{P}[F_a(X,Z) = -1] = 0$.

Finally, note that $\varepsilon>0$ is arbitrary, and thus $F_1=F_2$ a.s.. This establishes the desired uniqueness. 

%

\item[(ii)] As $L_p^{(\mathbf{u}^*_p)}$ and $\tilde{L}_p^{(\mathbf{u}^*_p)}$ are close in weak topology, it suffices to work with $\tilde{L}_p^{(\mathbf{u}^*_p)}$. With $\tilde{\mathbf{u}}_{p,\delta,K}$ as in \eqref{eq:u_defn}, for any $\varepsilon>0$, Lemma \ref{lemma:lower_bound_argument} and \eqref{eq:separate} imply 
\begin{align}
\limsup_{K \to \infty} \limsup_{\delta \to 0} \limsup_{p \to \infty} \mathbb{P}[\| \tilde{\mathbf{u}}_{p,\delta,K} - \mathbf{u}^*_p \|_2^2 > p \varepsilon | \mathbf{X}] =0. \nonumber 
\end{align} 
Recalling the 2-Wasserstein distance $d(\cdot, \cdot)$ from Lemma \ref{lemma:tightness}, we have, 
\begin{align}
 d( \tilde{L}_p^{(\mathbf{u}_p^*)} , \tilde{L}_p^{(\mathbf{u}_{p,\delta,K})} )  \leq  \frac{1}{p} \| \tilde{\mathbf{u}}_{p,\delta,K} - \mathbf{u}^*_p \|_2^2, \nonumber 
\end{align} 
and so it suffices to look at $\tilde{L}_p^{(\mathbf{u}_{p,\delta,K})}$. Let $\eta : \mathbb{R}^3 \to \mathbb{R}$ be bounded continuous. Then 
\begin{align}
\mathbb{E}_{\tilde{L}_p^{(\mathbf{u}_{p,\delta,K})} } [\eta(X,Z,U) ] = \frac{1}{p} \sum_{i=1}^{p} \eta(V_i , \xi_i, \tilde{u}_{p,\delta,K}(i) ) = \frac{1}{p} \sum_{i=1}^{p} \mathbb{E}[\eta(V_i , \xi_i, \tilde{u}_{p,\delta,K}(i) )] + o_P(1) , \nonumber
\end{align} 
where the last equality follows from the law of large numbers and the boundedness of $\eta$. Here, the $o_P(1)$ term converges to zero in probability under the joint distribution of $\{(V_i,\xi_i): 1 \leq i \leq p\}$. Finally, 
\begin{align}
\Big| \frac{1}{p} \sum_{i=1}^{p} \mathbb{E}[\eta(V_i , \xi_i, \tilde{u}_{p,\delta,K}(i) )]-  \mathbb{E}[\eta(X , Z, F(X, Z) \mathbf{1}(  F(X,Z) > -1 + \delta  ) )] \Big| \leq \mathbb{P}[F(X,Z) \leq -1+\delta]. \nonumber 
\end{align} 
This gives, for any $\varepsilon>0$, 
\begin{align}
\limsup_{K \to \infty} \limsup_{\delta \to 0} \limsup_{p \to \infty} \mathbb{P}\Big[  \Big| \mathbb{E}_{\tilde{L}_p^{(\mathbf{u}_{p,\delta,K})} } [\eta(X,Z,U) ] - \mathbb{E}[\eta(X , Z, F(X, Z) )] \Big| > \varepsilon | \mathbf{X}\Big] = 0, \nonumber 
\end{align} 
where we again use that $\mathbb{P}[F(X,Z) = -1] =0$. This completes the proof of Part (ii). 

%
%
%
%

\item[(iii)] For any $p \geq 1$, the function $\widetilde{M}_p(\cdot)$ is upper semicontinuous on $[-1,1]^p$, and thus attains its maximum. Let $\{ \hat{\mathbf{u}}_p : p \geq 1\}$ denote any sequence of global maximizers of $\widetilde{M}_p(\cdot)$. Lemma \ref{lem:off_diagonal} and the separation condition \eqref{eq:separate} together imply that $\frac{1}{p} \| \mathbf{u}_p^* - \hat{\mathbf{u}}_p\|_2^2 \stackrel{P| \mathbf{X}}{\to} 0$. In turn, this implies
\begin{align}
d(\tilde{L}_p^{(\mathbf{u}_p^*)}, \tilde{L}_p^{(\hat{\mathbf{u}}_p)}) \stackrel{P|\mathbf{X}}{\to} 0, \label{eq:wasserstein}
\end{align}
  where $d(\cdot, \cdot)$ denotes the $2$-Wasserstein distance. Thus $\tilde{L}_p^{(\hat{\mathbf{u}}_p)}$ converges weakly in probability to $\nu^*:=(X,Z, F^*(X,Z))$. 
  
 Next, differentiating $\widetilde{M}_p(\cdot)$ at $\mathbf{u} \in (-1,1)^p$, we obtain 
\begin{align}
 \nabla \widetilde{M}_p(\mathbf{u}) = \Big[ \frac{1}{\sigma^2} \Big( - A_p\mathbf{u}  + A_p \bbeta_0 + D_p \bbeta_0 + \textrm{diag}(\sqrt{D_p(i,i)}) \mathbf{\xi} \Big) - h(\mathbf{u}, \mathbf{d}) \Big], \nonumber 
\end{align} 
where $\xi = (\xi_1, \cdots, \xi_p)^{\mathrm{T}}$, and $h(\cdot, \cdot)$ is the vector obtained upon applying $h$ coordinate-wise to the two vectors. 
Note that if $u_i \to +1 $, then $h(u_i,d_i) \to  \infty$, and thus $\frac{\partial}{\partial u_i} \widetilde{M}_p(\mathbf{u}, \mathbf{d}) \to - \infty$. Consequently, $\hat{u}_i < 1$. A similar argument implies $\hat{u}_i > -1$, and thus  $\hat{\mathbf{u}}_p \in (-1,1)^p$. This immediately implies that $\hat{\mathbf{u}}_p$ is a critical point of $\widetilde{M}_p$, and satisfies the fixed point equation 
\begin{align}
\hat{\mathbf{u}}_p = \dot{c}\Big( \frac{1}{\sigma^2} \Big( - A_p \hat{\mathbf{u}}_p + A_p \bbeta_0 + D_p \bbeta_0 +  \textrm{diag}(\sqrt{D_p(i,i)}) \mathbf{\xi}   \Big), \mathbf{d}   \Big). \label{eq:fixed_point} 
\end{align} 
Let $f_1: [0,1] \to \mathbb{R}$ and $f_2: \mathbb{R} \to \mathbb{R}$ be bounded continuous functions. Recalling the function $m$ from Lemma \ref{lem:map_cont}  and setting $g_p = W_{p A_p} \cdot w_{\bbeta_0} + w_{\bbeta_0} w_{\mathbf{D}}$, we have, 
\begin{align}
&\mathbb{E}_{\tilde{L}_p^{(\hat{\mathbf{u}}_p)}}[f_1(X) f_2(Z) U] \nonumber\\
=& \mathbb{E}_{\tilde{L}_p^{(\hat{\mathbf{u}}_p)}}\Big[ f_1(X) f_2(Z)  \dot{c} \Big( \frac{1}{\sigma^2} \Big( - m(\tilde{L}_p^{(\hat{\mathbf{u}}_p)}, W_{p A_p},X ) + g_p(X) + \sqrt{w_{\mathbf{D}}(X)}Z   \Big) , w_{\mathbf{d}}(X)   \Big)    \Big] \nonumber \\
=&\mathbb{E}_{\tilde{L}_p^{(\hat{\mathbf{u}}_p)}}\Big[f_1(X) f_2(Z)  \dot{c} \Big( \frac{1}{\sigma^2} \Big( - m(\tilde{L}_p^{(\hat{\mathbf{u}}_p)}, W_{p A_p},X ) + g(X) + \sqrt{\psi(X)}Z   \Big) , w_{\mathbf{d}}(X)    \Big)  \Big] + \mathcal{E}_p \label{eq:approx1}   
\end{align} 
where we use $\ddot{c}(\theta,d) \leq 1$, and observe that 
\begin{align}
|\mathcal{E}_p| \lesssim  \mathbb{E}_{\tilde{L}_p^{(\hat{\mathbf{u}}_p)}}[| (\sqrt{w_{\mathbf{D}}(X)} - \sqrt{\psi(X)} )Z| ] + \|g_p - g\|_1 \lesssim \| w_{\mathbf{D}} - \psi\|_1 + \|g_p - g\|_1 = o(1). \label{eq:approx2}  
\end{align} 

Define the good event
\begin{align}
\mathscr{E}(\varepsilon) = \{|m(\tilde{L}_p^{(\hat{\mathbf{u}}_p)}, W_{pA_p}, X) - m(\tilde{L}_p^{(\hat{\mathbf{u}}_p)}, W,X)| \leq \varepsilon\} \cap \{ | w_{\mathbf{D}}(X) - \psi(X) | \leq \varepsilon\}. \nonumber 
\end{align} 
Then, upon observing that $\sup_{\gamma_1, \gamma_2} |\frac{\partial^2}{\partial \gamma_1 \partial \gamma_2}  c(\gamma_1, \gamma_2)| \lesssim 1 $ and $ \sup_{\gamma_1, \gamma_2} |\frac{\partial^2}{\partial \gamma_1^2} c(\gamma_1, \gamma_2) | \lesssim 1$, we have,  on the event $\mathscr{E}(\varepsilon)$,
\begin{align}
\Big|  \dot{c} \Big( \frac{1}{\sigma^2} \Big( - m(\tilde{L}_p^{(\hat{\mathbf{u}}_p)}, W_{p A_p},X )  + g(X) + \mathbb{E}[\sqrt{\psi(X)}Z]   \Big) , w_{\mathbf{d}}(X)   \Big)  - \nonumber \\
-  \dot{c} \Big( \frac{1}{\sigma^2} \Big( - m(\tilde{L}_p^{(\hat{\mathbf{u}}_p)}, W,X )  + g(X) + \mathbb{E}[\sqrt{\psi(X)}Z]   \Big) , \frac{\psi(X)} {\sigma^2}  \Big)  \Big| \lesssim \varepsilon. \nonumber
\end{align}

Thus we have, 
\begin{align}\label{eq:approx3}
\begin{split} 
&\Big| \mathbb{E}_{\tilde{L}_p^{(\hat{\mathbf{u}}_p)}}\Big[ f_1(X) f_2(Z)  \dot{c} \Big( \frac{1}{\sigma^2} \Big( - m(\tilde{L}_p^{(\hat{\mathbf{u}}_p)}, W_{p A_p},X ) + g(X) + \sqrt{\psi(X)}Z   \Big) , w_{\mathbf{d}}(X)  \Big)   \Big] \\
&-  \mathbb{E}_{\tilde{L}_p^{(\hat{\mathbf{u}}_p)}}\Big[ f_1(X) f_2(Z)  \dot{c} \Big( \frac{1}{\sigma^2} \Big( - m(\tilde{L}_p^{(\hat{\mathbf{u}}_p)}, W,X ) + g(X) + \sqrt{\psi(X)}Z   \Big) , \frac{\psi(X)}{\sigma^2}  \Big)   \Big]  \Big|   \\
&\lesssim \varepsilon + \mathbb{P}[\mathscr{E}(\varepsilon)^c]   \\
&\lesssim  \varepsilon + \frac{2}{\varepsilon} \|W - W_{pA_p}\|_{\square} + \mathbb{P}\Big[| w_{\mathbf{d}}(X) - \frac{ \psi(X)}{\sigma^2}| > \varepsilon \Big] = O(\varepsilon) + o(1), 
\end{split}  
\end{align}
where the last inequality uses Lemma \ref{lem:map_cont} part (i).

\noindent

Finally, Lemma \ref{lem:map_cont} part (ii) implies that 
\begin{align}
 &\mathbb{E}_{\tilde{L}_p^{(\hat{\mathbf{u}}_p)}}\Big[ f_1(X) f_2(Z)  \dot{c} \Big( \frac{1}{\sigma^2} \Big( -m(\tilde{L}_p^{(\hat{\mathbf{u}}_p)}, W,X ) + g(X) + \sqrt{\psi(X)}Z   \Big) , \frac{\psi(X)}{\sigma^2}    \Big)  \Big] \nonumber\\
 =& \mathbb{E}_{\tilde{L}_p^{(\hat{\mathbf{u}}_p)}}\Big[ f_1(X) f_2(Z)   \dot{c} \Big( \frac{1}{\sigma^2} \Big( - m(\nu^*, W,X ) + g(X) + \sqrt{\psi(X)}Z   \Big) , \frac{\psi(X)}{\sigma^2}    \Big)  \Big] +o(1). \label{eq:approx4}  
\end{align} 



Combining \eqref{eq:approx1}, \eqref{eq:approx2}, \eqref{eq:approx3}, \eqref{eq:approx4}, we have, 
\begin{align}
&\mathbb{E}_{\tilde{L}_p^{(\hat{\mathbf{u}}_p)}}[f_1(X) f_2(Z) U] \nonumber\\ 
=&  \mathbb{E}_{\tilde{L}_p^{(\hat{\mathbf{u}}_p)}}\Big[ f_1(X) f_2(Z)   \dot{c} \Big( \frac{1}{\sigma^2} \Big( - m(\nu^*, W,X ) + g(X) + \sqrt{\psi(X)}Z   \Big) , \frac{\psi(X)}{\sigma^2}     \Big)  \Big]  + o(1) + O(\varepsilon). \nonumber 
\end{align} 
Sending $p \to \infty$, and then $\varepsilon \to 0$, we obtain 
\begin{align}
&\mathbb{E}_{\nu^*}[f_1(X) f_2(Z) U] \nonumber \\
=&  \mathbb{E}_{\nu^*}\Big[ f_1(X) f_2(Z)   \dot{c} \Big( \frac{1}{\sigma^2} \Big( - m(\nu^*, W,X ) + g(X) + \sqrt{\psi(X)}Z   \Big) , \frac{\psi(X)}{\sigma^2}   \Big)  \Big]. \nonumber 
\end{align} 
The desired conclusion follows upon recalling that $(X,Z,F^*(X,Z)) \sim \nu^*$. 
\end{itemize}
\end{proof} 

We next turn to the proof of Lemma~\ref{lemma:pure_state}. To this end, we require the following auxiliary lemma. The proof is deferred to the Appendix. 

\begin{lemma}\label{lem:concave}

Suppose $r_p : [-1,1]^p \to \mathbb{R} \cup \{-\infty\}$ is upper semi-continuous on $[-1,1]^p$, finite and differentiable on $(-1,1)^p$,  and 
$$\sup_{{\bf u}\in (-1,1)^p}\lambda_{\min}(H_p({\bf x}))\le -\eta,$$ where $H_p(\cdot)$ is the Hessian of $r_p(\cdot)$. Then there exists a unique global maximizer ${\bf x}_0\in [-1,1]^p$, and further for any ${\bf x}\in [-1,1]^p$ we have 
\begin{align*}
r_p({\bf x}_0)-r_p({\bf x})\ge \frac{\eta}{2}\lVert {\bf x}-{\bf x}_0\rVert_2^2.
\end{align*}
\end{lemma}

\begin{proof}[Proof of Lemma~\ref{lemma:pure_state}]

\textit{Proof of (i)} 
The equation \eqref{eq:ps} implies that the event 
$$\sup_{{\bf u}\in [-1,1]^p:\lVert {\bf u}-{\bf u}_p^*\rVert_2^2>p \varepsilon }\frac{1}{p}\{M_p({\bf u})-M_p({\bf u}_p^*)\}<-\varepsilon \lambda$$
occurs with probability $1- o(1)$.

\textit{Proof of (ii)(a)} 
Differentiating $M_p(\cdot)$, twice, the Hessian is given by 
\begin{align}\label{eq:hessian}
H_p = - \frac{1}{\sigma^2} A_p - \Delta_p,\quad 
\Delta_p(i,i) = \frac{1}{\ddot{c}(h(u_i,d_i),d_i)}. 
\end{align}
Note that $\ddot{c}(h(u,d), d)$ is the variance of a random variable supported on $[-1,1]$, and thus $\ddot{c}(h(u,d),d) \leq 1$. Under the assumption of (i), there exists $\delta>0$ such that for all $p$ large enough and all ${\bf x} \in (-1,1)^p$, 
\begin{align}\label{eq:lower_eigen}
{\bf x}^{\mathrm{T}} H_p(u) {\bf x} = - \frac{1}{\sigma^2}{\bf x}^{\mathrm{T}} A_p {\bf x} - {\bf x}^{\mathrm{}} \Delta_p {\bf x} \leq - \delta \|{\bf x}\|_2^2.  
\end{align}
It then follows from Lemma \ref{lem:concave} that there exists ${\bf u}^*\in [-1,1]^p$ which is a unique optimizer of $M_p(\cdot)$, and further, for any ${\bf u}\in [-1,1]^p$ we have
$$M_p({\bf u}^*)-M_p({\bf u})\ge \frac{\delta}{2} \lVert {\bf u}-{\bf u}^*\rVert_2^2.$$
This verifies \eqref{eq:ps}. 

\textit{Proof of (ii)(b)} 
To begin, note that the prior is absolutely continuous with respect to the Lebesgue measure on $[-1,1]$ and thus $D_{KL}(\pi_{\infty}\| \pi) = D_{KL}(\pi_{-\infty} \| \pi) = \infty$. Thus to maximize $M_p(\cdot)$, it suffices to restrict to $(-1,1)^p$. As in \eqref{eq:lower_eigen}, it suffices to bound the lower eigenvalue of the Hessian, uniformly in $(-1,1)^p$, for which using \eqref{eq:hessian} and the fact that $\lambda_{\min}(\mathbf{X}^{\mathrm{T}}\mathbf{X})$ is bounded away from $0$ 
it suffices to show that $\Delta_p(i,i)\ge d_i$. 
By GHS inequality \cite{ellis1976ghs}, $\ddot{c}(h(u,d_i),d_i) \leq \ddot{c}(0,d_i)$, and thus the desired inequality follows, once we establish that 
\begin{align}\label{eq:ghs_conclude}
\sup_{d\geq 0} d \ddot{c}(0,d) \leq 1.
\end{align}
Consider now a scale family of parametric distributions $\{\mathcal{P}_{\theta}: \theta\geq 0\}$ with 
\begin{align}
\frac{\mathrm{d} \mathcal{P}_{\theta}}{\mathrm{d}x} \propto  \exp(- \theta V(x)) \exp(-dx^2/2). \nonumber 
\end{align}
Since $V(\cdot)$ is even, it follows that the first moment under $\mathcal{P}_\theta$ is $0$ for all $\theta$. Also, since $V(\cdot)$ is an increasing, this family has monotone likelihood ratio in $T(x)=|x|$, and thus the second moment under the law $\mathcal{P}_{\theta}$ is decreasing in $\theta$ for $\theta>0$, and so $$\ddot{c}(0,d)={\sf  Var}_{\theta=1}(Z)\le {\sf Var}_{\theta=0}(Z).$$
Proceeding to bound the RHS of the above display, for $d>0$ we have
\begin{align}
d {\sf Var}_{\theta=0}(Z)= \frac{\int_{-1}^{1} d z^2 \exp(-d z^2/2) \,\mathrm{d}z }{ \int_{-1}^{1} \exp(-dz^2/2) \,\mathrm{d}z}. \nonumber 
\end{align}
We substitute $\sqrt{d} z =t$, so that 
\begin{align}
d{\sf Var}_{\theta=0}(Z) = \frac{\int_{-\sqrt{d}}^{\sqrt{d}}t^2 \exp(-t^2/2) \,\mathrm{d}t}{ \int_{-\sqrt{d}}^{\sqrt{d}} \exp(-t^2/2) \,\mathrm{d}t}. \nonumber 
\end{align}
This is exactly the variance of a truncated standard Gaussian distribution, truncated to the interval $[-\sqrt{d}, \sqrt{d}]$. Indeed, the truncated variance  is $1- \frac{2\sqrt{d} \phi(\sqrt{d})}{2 \Phi(\sqrt{d}) -1} <1$ \cite{kotz1971univariate}, where $\phi(\cdot)$ and $\Phi(\cdot)$ represent the pdf and cdf of the standard Gaussian distribution. 
\end{proof}

\bibliographystyle{alpha}
\bibliography{meanfieldrefs}

\appendix

\section{Some technical lemmas} 
We collect some basic results about exponential families in the first subsection. We prove Lemma \ref{lem:concave} in the next subsection. 

\subsection{Results on exponential families} 

We prove Lemmas \ref{lemma:properties}, \ref{lemma:G_stability} and \ref{lemma:elementary} in this section. 

\begin{proof}[Proof of Lemma \ref{lemma:elementary}]

\begin{itemize}

\item[(i)] This follows by direct computation (see e.g. \cite{lehmann2006theory}). 

\item[(ii)] This follows by direct calculation. 

\item[(iii)] Recall from Definition \ref{def:exp} that 
\begin{align}
\frac{\mathrm{d} \psi_{\gamma} }{\mathrm{d}\pi}(z) = \exp\Big( \gamma_1 z - \frac{\gamma_2}{2} z^2 - c(\gamma)\Big). \nonumber  
\end{align}  
Without loss of generality, we consider the case $\gamma_{1,k} \to \infty$. The case $\gamma_{1,k} \to -\infty$ follows using the same argument, with obvious modifications. Observe that as $\gamma_{1,k} \to \infty$ and $\limsup_{k} |\gamma_{2,k}| < \infty $, for $k$ sufficiently large, the function $\gamma_{1,k} z  - \frac{\gamma_{2,k} }{2} z^2$ is increasing on $[-1,1]$. This implies, for any $\varepsilon>0$, 
\begin{align}
\int_{-1}^{1-\varepsilon} \exp\Big( \gamma_{1,k} z - \frac{\gamma_{2,k} }{2}z^2  \Big) \mathrm{d}\pi(z) \leq \exp\Big( \gamma_{1,k} (1- \varepsilon) - \frac{\gamma_{2,k}}{2}(1-\varepsilon)^2  \Big)  \pi([-1, 1-\varepsilon]). \nonumber 
\end{align} 
On the other hand, 
\begin{align}
\exp(c(\gamma_k)) \geq \int_{1-\frac{\varepsilon}{2}}^{1} \exp\Big( \gamma_{1,k} z - \frac{\gamma_{2,k} }{2}z^2  \Big)\mathrm{d}\pi(z) \geq  \exp\Big( \gamma_{1,k} (1- \frac{\varepsilon}{2}) - \frac{\gamma_{2,k}}{2} (1- \frac{\varepsilon}{2})^2  \Big) \pi([1 - \frac{\varepsilon}{2}, 1]). \nonumber 
\end{align} 
On taking ratios of the above two displays we get
\begin{align*}
\pi_{\gamma_k}([-1,1-\varepsilon])\le \exp\Big(-\frac{\varepsilon}{2} \gamma_{1,k}+\frac{1}{2}\gamma_{2,k}^2\Big) \frac{\pi([-1,1-\varepsilon])}{\pi([1-\frac{\varepsilon}{2}, 1])}
\end{align*}
Also note that the ratio of probabilities under $\pi$ in the RHS is positive, as $\pm 1$ is in the support of $\pi$.
The desired conclusion now follows upon letting $k\to\infty$, upon noting that $\gamma_{1,k}\to\infty$, and $\gamma_{2,k}$ stays bounded.


\item[(iv)] 
The proof follows directly from the lower semicontinuity of KL divergence under weak convergence \cite[Theorem 1]{posner1975random}.

\item[(v)] By definition, 
\begin{align}
H(x,y) = \int_{[-1,1]} z^r \exp\Big( h(x,y) z - \frac{y}{2} z^2 - c(h(x,y),y) \Big) \mathrm{d}\pi(z). \nonumber 
\end{align} 
For $x \in \{-1, 1\}$, $H(x, \cdot)$ is independent of $y$, and thus $\frac{\partial H(x,y)}{\partial y} =0$. We assume henceforth that $x \in (-1,1)$. Differentiating, we obtain, 
\begin{align}
\frac{\partial H(x,y)}{\partial y} =& \int_{[-1,1]} z^r  \Big[ z \frac{\partial h(x,y)}{\partial y} -  \frac{z^2}{2} - \Big( \dot{c}(h(x,y),y) \frac{\partial h(x,y)}{\partial y} + \frac{\partial c}{\partial \gamma_2} \Big|_{(h(x,y),y)}  \Big)  \Big]  \mathrm{d}\pi_{(h(x,y),y)} (z). \nonumber\\
 =& \int_{[-1,1]} z^r  \Big[ (z-x) \frac{\partial h(x,y)}{\partial y} -  \frac{z^2}{2} -  \frac{\partial c}{\partial \gamma_2} \Big|_{(h(x,y),y)}  \Big]  \mathrm{d}\pi_{(h(x,y),y)} (z), \nonumber 
\end{align} 
where the last equality uses $\dot{c}(h(x,y),y)) = x$. 
Therefore, 
\begin{align}
\Big| \frac{\partial H(x,y)}{\partial y} \Big| \leq \Big|\frac{\partial h(x,y)}{\partial y}  \Big|   \mathbb{E}_{\pi_{(h(x,y),y)}}[|Z-x|] + \frac{1}{2} + \frac{1}{2} \leq  \Big|\frac{\partial h(x,y)}{\partial y} \Big| + 1, \label{eq:H_stability_int} 
\end{align} 
where we use $\mathbb{E}_{\pi_{(h(x,y),y)}}[|Z|^r] \leq 1$ for all $ r \geq 1$. Now, differentiating $\dot{c}(h(x,y),y) = x$ in $y$, we have, 
\begin{align} 
\frac{\partial h(x,y)}{\partial y}   =& - \frac{\frac{\partial^2 c}{\partial \gamma_1 \partial \gamma_2} (h(x,y),y) }{\ddot{c}(h(x,y),y)}  \nonumber \\
=& \frac{1}{2} \frac{\mathrm{Cov}_{\pi_{(h(x,y),y)}}(Z,Z^2) }{\mathrm{Var}_{\pi_{(h(x,y),y)}}(Z)} \nonumber \\
=& \frac{1}{2} \frac{\mathbb{E}_{(h(x,y),y)}[(Z-x)(Z^2- x^2)]}{\mathbb{E}_{(h(x,y),y)}[(Z-x)^2] } \nonumber \\
=& \frac{1}{2} \frac{\mathbb{E}_{(h(x,y),y)}[(Z-x)^2 (Z+x)]}{\mathbb{E}_{(h(x,y),y)}[(Z-x)^2] }.  \nonumber
\end{align} 
This implies $\sup_{x \in (-1,1), y \in \mathbb{R}} \Big| \frac{\partial h(x,y)}{\partial y} \Big| \leq 1$, where we use the trivial bound $|Z+x| \leq 2$. The desired conclusion follows by plugging this back into \eqref{eq:H_stability_int}. 

\end{itemize}

\end{proof}

\begin{proof}[Proof of Lemma \ref{lemma:properties}]

\begin{itemize}
\item[(a)] Lemma \ref{lemma:elementary} Part (a)  implies that $\ddot{c}(\gamma_1, \gamma_2) = \mathrm{Var}_{\psi_{\gamma}} (Z)  > 0$ for all $\gamma_1, \gamma_2$. Thus $\dot{c}(\gamma_1, \gamma_2)$ is strictly increasing in $\gamma_1$.  For every $\gamma_2 \in \mathbb{R}$, as $\gamma_1 \to \pm \infty$, $\psi_{\gamma} \stackrel{w}{\to} \psi_{\pm \infty}$ using Lemma \ref{lemma:elementary} Part (c). The desired conclusion follows on noting that 
\begin{align}
\dot{c}(\gamma_1, \gamma_2) = \mathbb{E}_{\psi_\gamma}[Z] \stackrel{\gamma_1 \to \pm \infty}{\to } \pm 1 \nonumber 
\end{align} 
by the Dominated Convergence Theorem. 

\item[(b)] Using Part (a), we know that $\dot{c}(\gamma_1, \gamma_2)$ is strictly increasing in $\gamma_1$, and continuous. Further, if $\gamma_1 \to \pm \infty$, $\dot{c}(\gamma_1, \gamma_2) \to \pm 1$. Thus for any $t \in (-1,1)$, the existence of $h(t,\gamma_2)$ follows by the Intermediate Value Theorem. Finally, we argue that for any fixed $\gamma_2 \in \mathbb{R}$,  as $ t \to 1$,  $h(t, \gamma_2) \to \infty$. The case $ t \to -1$ is similar, and thus omitted. 

We complete the proof by contradiction. Suppose, if possible, that $M:=\lim_{ t \to 1} h(t, \gamma_2) < \infty$. In this case, as $-1$ is in the support of $\pi$, $\pi([-1, 0])>0$. This implies that 
\begin{align}
t=  \dot{c}(h(t, \gamma_2) , \gamma_2) &=   \int_{[-1,1]} z \exp\Big( h(t, \gamma_2)  z - \frac{\gamma_{2} }{2}z^2 - c(h(t,\gamma_2),\gamma_2)  \Big) \mathrm{d}\pi(z) \nonumber \\
&= \int_{[-1,0]}  z \exp\Big( h(t, \gamma_2)  z - \frac{\gamma_{2} }{2}z^2 - c(h(t,\gamma_2),\gamma_2)  \Big) \mathrm{d}\pi(z) +  \nonumber \\
& \int_{(0,1]} z \exp\Big( h(t, \gamma_2)  z - \frac{\gamma_{2} }{2}z^2 - c(h(t,\gamma_2),\gamma_2)  \Big) \mathrm{d}\pi(z) \nonumber \\
& \leq \pi_{(h(t,\gamma_2), \gamma_2)}((0,1]). \nonumber 
\end{align} 
To complete the argument, it suffices to show that $\liminf_{t \to 1} \pi_{(h(t,\gamma_2), \gamma_2)}([-1,0])>0$. But this follows on noting that the function $ t \to c( h(t,\gamma_2) , \gamma_2)$ is non-decreasing on $[0,1]$, and thus 
\begin{align}
\pi_{(h(t,\gamma_2), \gamma_2)}([-1,0]) \geq \exp( -M - \frac{\gamma_2}{2} - c(M, \gamma_2)) \pi([-1,0]). \nonumber 
\end{align} 

\end{itemize}
\end{proof}

\begin{proof}[Proof of Lemma \ref{lemma:G_stability}]
The lemma follows by direct computation. First, note that 
\begin{align}
\frac{\partial G(u,d)}{\partial u}  = h(u,d) + u \frac{\partial h}{\partial u}(u,d) - u \dot{c}(h(u,d),d) = h(u,d), \nonumber 
\end{align} 
where the last equality follows upon noting that $\dot{c}(h(u,d),d) = u$. For the second derivative, note that 
\begin{align*}
\frac{\partial G(u,d)}{\partial d}=&u\frac{\partial h(u,d)}{\partial d}-\dot{c}(h(u,d),d)\frac{\partial h(u,d)}{\partial d}-\frac{\partial }{\partial \gamma_2}c(\gamma_1,\gamma_2)\Big|_{(h(u,d),d)} + \frac{\partial }{\partial \gamma_2}c(\gamma_1,\gamma_2)\Big|_{(0,d)} \\
=&-\frac{\partial }{\partial \gamma_2}c(\gamma_1,\gamma_2)\Big|_{(h(u,d),d)} + \frac{\partial }{\partial \gamma_2}c(\gamma_1,\gamma_2)\Big|_{(0,d)}  \\
=& \frac{1}{2}\int_{[-1,1]} z^2 d\pi_{(h(u,d),d)}(z)dz - \frac{1}{2}\int_{[-1,1]} z^2 d\pi_{(0,d)}(z)dz .
\end{align*}
Finally, note that 
\begin{align}
\frac{\partial^2 G(u,d)}{\partial^2 u}  = \frac{\partial h(u,d)}{\partial u} = \frac{1}{\ddot{c}(h(u,d),d)} >0 \nonumber 
\end{align} 
where the last equality follows from differentiating the equation $\dot{c}(h(u,d),d) = u$. 
\end{proof} 

\subsection{Proof of Lemma~\ref{lem:concave}}

\begin{proof}
Since $r_p(\cdot)$ is upper semi-continuous there exists a global maximizer in $[-1,1]^p$, say $\widetilde{\bf x}_0$. Fixing any ${\bf x}$ in $(-1,1)^p$, consider the function $f(t):=r_p( (1-t) \widetilde{\bf x}_0+t {\bf x})$. Then $f$ is twice differentiable on $(0,1)$, and
$$f''(t)=(\widetilde{\bf x}_0-{\bf x})^{\mathrm{T}}H_p( (1-t) \widetilde{\bf x}_0+ t{\bf x})(\widetilde{\bf x}_0-{\bf x})\le- \eta \lVert \widetilde{\bf x}_0-{\bf x}\rVert_2^2.$$
Consequently, for any $\varepsilon\in (0,1)$ Taylor's expansion implies
$$f(1)\le f(\varepsilon)+(1-\varepsilon)f'(\varepsilon)-\frac{1}{2}\eta (1-\varepsilon)^2 \lVert \widetilde{\bf x}_0-{\bf x}\rVert_2^2.$$
Taking limits as $\varepsilon\to0$ we get
$$f(1)\le f(0)+f'(0+)-\frac{1}{2}\eta  \lVert \widetilde{\bf x}_0-{\bf x}\rVert_2^2\le  f(0)-\frac{1}{2}\eta  \lVert \widetilde{\bf x}_0-{\bf x}\rVert_2^2,$$
where the last inequality uses the fact that $f'(0+)\le 0$, as $t=0$ is a global maximizer of $f$.  The last display is equivalent to
$$r_p({\bf x})\le r_p(\widetilde{\bf x}_0)-\frac{1}{2}\eta  \lVert \widetilde{\bf x}_0-{\bf x}\rVert_2^2.$$
This inequality then extends to $x \in [-1,1]^{p}$ by upper semi-continuity of $r_p$. 
\end{proof} 

\section{Stability Estimates for functionals} 
We prove Lemma \ref{lem:rate2} and \ref{lem:map_cont} in this section. 
\subsection{Proof of Lemma \ref{lem:rate2}} 
We start with the proof of Part (i). 
Define the functional $T_{W, \phi} : \tilde{\mathscr{F}}_{2,4}  \to \mathbb{R}$ 
\begin{align*}
T_{W,\phi,\psi}(\nu) :=-\frac{1}{2} \mathbb{E}[W(X_1, X_2) U_1 U_2] + \mathbb{E}[U_1 \phi(X_1)] + \mathbb{E}[\sqrt{\psi(X_1)} U_1 Z_1] .
\end{align*}
where $(X_1, Z_1, U_1) , (X_2, Z_2, U_2) \sim \nu$ are iid. Similarly define $I_{\psi} : \tilde{\mathscr{F}}_{2,4}  \to \mathbb{R} \cup \{\infty\}$ 
\begin{align}
I_{\psi}(\nu) = \mathbb{E}\Big[G \Big(U,  \frac{\psi(X)}{\sigma^2} \Big) \Big]. \nonumber 
\end{align}

 We observe that $\tilde{\mathcal{G}}_{W,\phi,\psi}(\nu) = \frac{1}{\sigma^2} T_{W,\phi,\psi}(\nu) - I_{\psi}(\nu)$. Further,  
%
%
%
%
\begin{align}
\sup_{\nu \in \tilde{\mathscr{F}}_{2,4}  } |T_{W,\phi,\psi }( \nu) - T_{\tilde{W} , \tilde{\phi}, \psi }(\nu) |\le   \frac{1}{2}\lVert W-\tilde{W}\rVert_\square+\lVert \phi-\tilde{\phi}\rVert_1, \label{eq:a_stability} \\
\sup_{\nu \in \tilde{\mathscr{F}}_{2,4}  } |T_{W,\phi,\psi }( \nu) - T_{W , \phi, \tilde{\psi} }(\nu) | \leq 2 \| \psi - \tilde{\psi} \|_1,  \label{eq:c_stability}  \\
\sup_{\nu \in \tilde{\mathscr{F}}_{2,4}  } | I_{\psi}(\nu) - I_{\tilde{\psi}}(\nu) |\le   \frac{1}{2 \sigma^2}  \lVert \psi-\tilde{\psi}\rVert_1.\label{eq:b_stability} 
\end{align}

%


Indeed, \eqref{eq:a_stability} is immediate from the definition of $T_{W,\phi,\psi} (\cdot)$, and \eqref{eq:b_stability} follows from Lemma \ref{lemma:G_stability} on noting that $\sup_{u \in (-1,1), d \in \mathbb{R}} |\frac{\partial G}{\partial d} (u,d)| \leq \frac{1}{2}$.  Finally,  \eqref{eq:c_stability} follows on noting that by  Cauchy-Schwarz inequality, 
\begin{align}
\mathbb{E}_{\nu}\Big[ | \sqrt{\psi (X)} Z - \sqrt{\tilde{\psi}(X) } Z| \Big] &\leq  \sqrt{ \mathbb{E}_{\nu} \Big[  Z^2  \Big]     \mathbb{E}_{\nu} \Big[ \Big(   \sqrt{\psi (X)} - \sqrt{ \tilde{\psi}(X)} \Big)^{2} \Big] } \leq 2 \| \psi - \tilde{\psi}\|_1,  \nonumber 
\end{align} 
where the last inequality uses $(\sqrt{a} - \sqrt{b})^2 \leq | a -b|$ and $\mathbb{E}_{\nu}[Z^2] \leq 4$. This completes the proof of Part (i). 

Next, we turn to the proof of Part (ii). 
%
Using the definition of cut norm (as in Definition \ref{def:cut}) and  Part (i), we have,
\begin{align*}
\sup_{\nu \in \tilde{\mathscr{F}}_{2,4}  }|\tilde{\mathcal{G}}_{W_k , W_k.\phi_k+ \phi_k \psi_k, \psi_k }(\nu) -\tilde{\mathcal{G}}_{W, W.\phi_k+ \phi_k \psi, \psi } (\nu) |\lesssim   \lVert W_k-W\rVert_\square + \|\psi_k - \psi\|_1 \to 0.
\end{align*}
 Using Part (i) again, it suffices to show that 
\begin{align}
W.\phi_k+ \phi_k \psi \stackrel{L^1}{\to} W.\phi + \phi \psi. \nonumber 
\end{align}

To this end, note that $|W.\phi_k + \phi_k (x) \psi (x) -W.\phi -\phi(x) \psi(x) |$ converges to $0$ in measure, and $$|W.\phi_k(x) + \phi_k(x) \psi (x) -W.\phi(x) - \phi(x) \psi(x)|\le 2\int_{[0,1]}|W(x,y)|dy + 2 \psi,$$ which is an integrable function. This completes the argument using DCT.

\subsection{Proof of Lemma~\ref{lem:map_cont}} 

\begin{proof}[Proof of Lemma~\ref{lem:map_cont}] 
\begin{itemize}
\item[(i)] 
For any $\varepsilon>0$, let $S^+(\varepsilon) = \{x: m(\mu, W,x) - m(\mu, W',x) > \varepsilon\}$. For $X \sim U([0,1])$, we have, 
\begin{align}
\mathbb{P}(X \in S^+(\varepsilon)) &\leq \frac{1}{\varepsilon}  \mathbb{E}\Big[ (m(\mu, W, X) - m(\mu, W', X)) \mathbf{1}_{S^+(\varepsilon)}(X)  \Big] \nonumber \\
&= \frac{1}{\varepsilon} \mathbb{E}_{X, X'}\Big[ (W(X,X') - W'(X,X')) \mathbf{1}_{S^+(\varepsilon)}(X) \mathbb{E}[U'|X']  \Big] \nonumber \\
&\leq  \frac{1}{\varepsilon} \|W - W'\|_{\square} . \nonumber 
\end{align} 
Setting $S^{-}(\varepsilon) = \{ x : m(\mu, W,x) - m(\mu, W',x) < - \varepsilon \}$, the same argument now yields that 
\begin{align}
\mathbb{P}[X \in S^-(\varepsilon)] \leq \frac{1}{\varepsilon} \|W - W'\|_{\square}. \nonumber 
\end{align}

\item[(ii)] Let $(X_p, Z_p, U_p) \sim \nu_p$ and $(X,Z,U) \sim \nu$. Using Skorokhod Embedding Theorem, we assume that $(X_p, Z_p, U_p) \to (X,Z,U)$ a.s. as $p \to \infty$. Since $\int |W(x,y)| \mathrm{d}x \mathrm{d}y < \infty$, for any $\varepsilon>0$ there exists $W'$ continuous such that $\|W - W' \|_1 \leq \varepsilon$. Then we have, 
\begin{align}
\mathbb{E}[|m(\mu, W,X) - m(\mu, W', X)| ] &\leq 
 \|W- W'\|_1 \leq \varepsilon. \nonumber 
\end{align}

Also, 
\begin{align}
|m(\nu_p,W',x) - m(\nu, W', x)| &= \Big| \mathbb{E}[ W'(x,X_p) U_p ] - \mathbb{E}[ W'(x,X) U ] \Big| \nonumber \\
&\leq \Big| \mathbb{E}[ W'(x,X_p) U_p ] - \mathbb{E}[ W'(x,X_p) U ] \Big| + \Big| \mathbb{E}[ W'(x,X_p) U ] - \mathbb{E}[ W'(x,X) U ] \Big| \nonumber\\
&\leq \|W'\|_{\infty} \mathbb{E}[|U_p - U|] + \sup_{x,y,z \in [0,1],|y-z| \leq |X-X_p|} \Big| W'(x,y) - W'(x,z)\Big| =o(1) \nonumber 
\end{align} 
using the uniform continuity of $W'$. The desired conclusion follows upon combining the two displays above. 
\end{itemize} 
\end{proof} 

\section{Proofs of Examples} 
\label{sec:example_proofs} 
We establish Corollaries \ref{cor:det1}-\ref{cor:erdos1} and \ref{cor:reg2}-\ref{cor:det2} in this section. Throughout this section $o_P(1)$ terms converge to zero in probability under the marginal distribution of the design matrix $\mathbf{X}$. 
\subsection{Accuracy of mean-field approximation} 
\begin{proof}[Proof of Corollary \ref{cor:det1}]
This is immediate from Theorem \ref{thm:main}.
\end{proof}

\begin{proof}[Proof of Corollary \ref{cor:reg1}]

With $A_p$ and $D_p$ denoting the off-diagonal and diagonal parts of the matrix ${\bf X}^{\mathrm{T}}{\bf X}$, to invoke Theorem \ref{thm:main} we need to verify that $A_p$ satisfies \eqref{eq:mean_field} and \eqref{eq:weak}, and the empirical measure $\frac{1}{p}\sum_{i=1}^p D_p(i,i)$ is uniformly integrable. We verify these conditions below:

Since $\{{\bf x}_i\}_{1\le i\le n}\stackrel{i.i.d.}{\sim} N(0,\Gamma_p)$, and $${\mathbf X}^{\mathrm T}{\mathbf X}=\frac{1}{n}\sum_{i=1}^n{\bf x}_i{\bf x}_i^{\mathrm T},$$
%
 invoking \cite[Proposition 2.1]{vershynin2012close} gives 
\begin{align}\label{eq:versh}
\Big\lVert{\bf X}^{\mathrm{T}}{\bf X}-\Gamma_p\Big\rVert_2=O_P\Big(\sqrt{\frac{p}{n}}\Big)=o_P(1),
\end{align}
where the last equality uses $p=o(n)$. Noting that
\begin{align*}
\lVert D_p-\Gamma_{p,{\rm diag}}\rVert_2\le \lVert {\bf X}^{\mathrm{T}}{\bf X}-\Gamma_{p}\rVert_2
\end{align*}
gives
\begin{align}\label{eq:versh2}
\lVert A_p-\Gamma_{p,{\rm off}}\rVert_2=o_P(1).
\end{align}
Consequently,  $A_p$ satisfies \eqref{eq:mean_field} with high probability, as ${\rm tr}(\Gamma_{p,{\rm off}}^2)=o(p)$. Further, since $\lVert \Gamma_p\rVert_2=O(1)$, it follows from \eqref{eq:versh} gives that $\lVert A_p\rVert_2=O_P(1)$. Finally, noting that $$\max_{1\le i\le p}|D_p(i,i)|\le \lVert {\bf X}^{\mathrm{T}}{\bf X}\rVert_2=\lVert \Gamma_p\rVert_2+o_P(1)$$
we have $\frac{1}{p}\sum_{i=1}^p \delta_{D_p(i,i)}$ is uniformly integrable with high probability. This completes the proof of the corollary.

%
%
%
%
%
\end{proof}

\begin{proof}[Proof of Corollary \ref{cor:erdos1}]

It suffices to verify the same conditions on the matrices $(A_p,D_p)$ as in Corollary \ref{cor:reg1}.

To this end, note that for any $i\ne j$ we have
\begin{align}\label{eq:rep}
A_p(i,j)=\frac{p}{n}\sum_{k=1}^nB(k,i) B(k,j),
\end{align}
and so
\begin{align*}
\E \sum_{i\ne j} A_p(i,j)^2=&\frac{p^2}{n^2} \sum_{i\ne j} \sum_{k,\ell=1}^n \E[B(k,i) B(k,j) B(\ell,i) B(\ell,j)]\\
=&\frac{p^2}{n^2}\sum_{i\ne j} \sum_{k\ne \ell} \E[B(k,i) B(k,j) B(\ell, i) B(\ell,j)]+\frac{p^2}{n^2}\sum_{i\ne j}\sum_{k=1}^n \E[B(k,i) B(k,j)]\\
\le &\frac{p^2}{n^2}\times \frac{n^2p^2\lambda^4}{p^4}+\frac{p^2}{n^2}\times\frac{np^2\lambda^2}{p^2}=\lambda^4+\frac{\lambda^2 p^2}{n},
\end{align*}
which is $o(p)$ as $p=o(n)$. This verifies \eqref{eq:mean_field}. 

Also, \eqref{eq:rep} gives
$$\frac{1}{p}\E \sum_{i,j=1}^p |A_p(i,j)|=\frac{1}{n}\sum_{k=1}^n \sum_{i\ne j} \E[ B(k,i) B(k,j)]\le \lambda^2,$$
and so \eqref{eq:weak} holds with high probability. Finally we have
$$\E\frac{1}{p}\sum_{i=1}^p |D_p(i,i)|^{2}=\frac{1}{p}\sum_{i=1}^p\E \Big(\frac{p}{n}\sum_{k=1}^n B(k,i)\Big)^2\le \frac{p^2}{n^2}\Big[\frac{n^2 \lambda^2}{p^2}+\frac{n\lambda}{p}\Big]=\lambda^2+o(1),$$
and so $\frac{1}{p}\sum_{i=1}^p D_p(i,i)$ is uniformly integrable. 
\end{proof}

\subsection{Limiting variational formula} 
\begin{proof}[Proof of Corollary \ref{cor:reg2}]

\begin{enumerate}
\item[(a)]
The desired conclusion follows from Theorem \ref{thm:var_conv_new}, once we can verify 
$$ d_{L_1}(w_{{\bf D}},1)=o_P(1),\quad d_\square(W_{pA_p},W)=o_P(1).$$
Here $\mathbf{D} = (D_p(1,1), \cdots, D_p(p,p))$ denotes the diagonal entries of ${\mathbf X}^{\mathrm T}{\mathbf X}$, and $A_p$ denotes the off diagonal part of ${\mathbf X}^{\mathrm T}{\mathbf X}$.
Proceeding to verify the above display, invoking \eqref{eq:versh} we have $D_p(i,i) =1+\frac{1}{p}G(i/p)+o_P(1)$, and so $w_{{\bf D}}\stackrel{L_1}{\to} 1$. Also, since $\Gamma_p(i,j)=\frac{1}{p}G(i/p) G(j/p)$ for all $i\ne j$, it follows 
that $d_\square(W_{p\Gamma_p},W)\to 0$, where we use the almost sure continuity of $G(.,.)$. Since \eqref{eq:versh2} gives $d_\square(W_{pA_p},W_{p\Gamma_p})=o_P(1)$, combining we get
$$d_\square(W_{pA_p},W)\le d_\square(W_{pA_p},W_{p\Gamma_p})+d_\square(W_{p\Gamma_p},W)=o_P(1),$$
 This completes the proof of part (a).

\item[(b)]

We begin by verifying condition \eqref{eq:separate}. To this effect, 
set $$\overline{M}_p({\bf u}):=-\frac{1}{2\sigma^2}\Big[{\bf u}^{\mathrm T}\Gamma_p{\bf u}-2{\mathbf z}^{\mathrm T}{\mathbf u}\Big]-\sum_{i=1}^p G(u_i,d_i),$$ 
and use the fact that $\max_{i\in [p]}\sum_{j\ne i}\Gamma_p(i,j)\le \lambda<\sigma^2$ along with part (a) of Lemma \ref{lemma:pure_state} to get the existence of ${\mathbf u}_p^*$  such that
\begin{align*}
\overline{M}_p({\bf u}_p^*)-\overline{M}_p({\bf u})\ge \lambda \lVert {\bf u}_p^*-{\bf u}\rVert_2^2.
\end{align*}
This in turn shows that for any $\varepsilon>0$ we have
\begin{align}\label{eq:separate_2}
\limsup_{p\to\infty} \sup_{{\bf u}:\lVert {\bf u}-{\bf u}_p^*\rVert_2^2>p\delta}\frac{1}{p}\Big\{\overline{M}_p({\bf u}_p^*)-\overline{M}_p({\bf u})\Big\}<0.
\end{align}
Using \eqref{eq:versh} gives
\begin{align}\label{eq:separate_3}
\frac{1}{p}\sup_{{\bf u}\in [-1,1]^p}\frac{1}{p}\Big[\overline{M}_p({\bf u})-\widetilde{M}_p({\bf u})\Big]=o_P(1).
\end{align}
Given \eqref{eq:separate_2}, and \eqref{eq:separate_3}, it follows that
\begin{align*}
\limsup_{p\to\infty} \sup_{{\bf u}:\lVert {\bf u}-{\bf u}_p^*\rVert_2^2>p\delta}\frac{1}{p}\Big\{{M}_p({\bf u}_p^*-{M}_p({\bf u})\}<0.
\end{align*}
Thus we have verified \eqref{eq:separate}. The desired conclusion then follows from Corollary \ref{cor:eg}.


Part (b)(ii) follows from part (b)(ii) of Lemma \ref{lemma:pure_state} and Corollary \ref{cor:eg}.

%
\end{enumerate}
\end{proof}

\begin{proof}[Proof of Corollary \ref{cor:erdos2}]
\begin{enumerate}
\item[(a)]

The desired conclusion follows from Theorem \ref{thm:var_conv_new}, once we can verify 
$$ d_{L_1}(w_{{\bf D}},\psi)=o_P(1),\quad d_\square(W_{pA_p},W)=o_P(1).$$
Proceeding to verify the above display, note that
$$D_p(i,i):=({\mathbf X}^{\mathrm T}{\mathbf X})_{ii}=\frac{p}{n\sigma^2}\sum_{k=1}^nB_{ki}.$$
Thus, setting $\widetilde{D}_i:=\frac{1}{n}\sum_{k=1}^n G(k/n,i/p)$, Using Chernoff bounds it follows that
$$\P\left(\max_{1\le i\le p}|D_p(i,i)-\widetilde{D}_i|>\delta\right)\le p e^{-c\delta^2 \frac{n}{p}}\le p e^{-c\delta^2 \sqrt{p}}\to 0,$$
and so it follows that
$$d_{L_1}(w_{{\bf D}},\psi)\le d_{L_1}(w_{{\bf D}},  w_{\widetilde{D},p})+d_{L_1}(w_{\widetilde{D},p},\psi)=o_P(1)$$
where the last equality uses the almost sure continuity of $G(.,.)$.
\\

It thus suffices to show that
$$d_\square(W_{pA_p}, W)\to0.$$
To this end, note that for any $i\ne j$ we have
$$A_p(i,j)=\frac{p}{n}\sum_{k=1}^nB(k,i) B(k,j).$$
Using Lemma \ref{lemma:concentration}, 
$$\P\Big(\max_{S,T\subseteq [p]}\Big|\frac{\sum_{i\in S, j\in T}pA_p(i,j)}{p^2}-\frac{\sum_{i\in S, j\in T}p\widetilde{A}_p(i,j)}{p^2} \Big|>\delta \Big)\le 2^{p} e^{-C  \sqrt{n \delta }}.$$
From this, using the condition $p=o(\sqrt{n})$  gives
$$d_\square(W_{pA_p},W_{p\widetilde{A}_p})=o_P(1),$$
which in turn gives
$$d_\square(W_{p{A}_p},W)\le d_\square(W_{pA_p},W_{p\widetilde{A}_p})+d_\square(W_{p\widetilde{A}_p},W)=o_P(1),$$
where the last equality again uses the almost sure continuity of $G(\dot,\dot)$.

\item[(b)]
(i)
Once again, the desired conclusion of part (b) follows from Corollary \ref{cor:eg}, once we verify condition \eqref{eq:separate}.

To this effect, fixing $\delta>0$ and setting $R_i:=\sum_{j\ne i}^p A_p(i,j)$, define a $p\times p$ matrix $B_{p,\delta}$ by setting
$$B_{p,\delta}(i,j):=A_p(i,j) 1\{R(i)\le \sigma^2(1-\delta), R(j)\le \sigma^2(1-\delta)\},$$
and note that
\begin{align*}
\sup_{{\bf u}\in [-1,1]^p} \Big|{\bf u}'A_p{\bf u}_p-{\bf u}'B_{p,\delta}{\bf u}\Big|\le \frac{2}{p}\sum_{i=1}^p R_i 1\{R_i>\sigma^2(1-\delta)\}.
\end{align*}
We now claim that there exists $\delta>0$ such that 
\begin{align}\label{eq:claim_2021}
\limsup_{p\to\infty} \frac{1}{p}\sum_{i=1}^p R_i 1\{R_i>\sigma^2(1-\delta)\}=0.
\end{align}
Given \eqref{eq:claim_2021}, it suffices to show that \eqref{eq:separate} holds for $\overline{M}_p:[-1,1]^p\mapsto \R$ defined by
$$\overline{M}_p({\bf u}):=-\frac{1}{2\sigma^2}\Big[{\bf u}^{\mathrm T}B_{p,\delta}{\bf u}-2{\mathbf z}^{\mathrm T}{\mathbf u}\Big]-\sum_{i=1}^p G(u_i,d_i).$$
But this is immediate from Lemma \ref{lemma:pure_state} part (a), on noting that
$$\max_{i\in [p]}\sum_{j\ne i}B_{p,\delta}(i,j)\le \sigma^2(1-\delta).$$
It thus remains to verify \eqref{eq:claim_2021}. To this effect, recall from part (a) that $W_{pA_p}$ converges to $W$ in the cut metric, which in turn implies $\frac{1}{p}\sum_{i=1}^p \delta_{R_i}$ converges weakly in probability to the law of $S(X)$, 
where $X\sim U[0,1]$ \cite[Theorem 2.16]{borgs2015consistent}. 
This gives$$\limsup_{p\to \infty} \frac{1}{p}\sum_{i=1}^p R_i 1\{R_i>\sigma^2(1-\delta)\}\le \E S(X) 1\{S(X)\ge \sigma^2(1-\delta)\}.$$
It thus suffices to show that there exists $\delta>0$ such that the RHS above is $0$. But this follows on noting that ${\rm ess sup}(S(X))<\sigma^2$. This completes the proof of part (i).

Part (b)(ii) follows from part (b)(ii) of Lemma \ref{lemma:pure_state} and Corollary \ref{cor:eg}, as before.

\end{enumerate}

\end{proof}

\begin{proof}[Proof of Corollary \ref{cor:det2}]
\begin{enumerate}
\item[(a)]

A direct calculation gives ${\mathbf X}^{\mathrm T}{\mathbf X}=\frac{1}{2}{\bf I}+A_p$, where 
\begin{align*}
A_p(i,j)=&\frac{1}{p}\text{ if }i\le \frac{p}{2}, j>\frac{p}{2} \text{ or }i>\frac{p}{2}, j\le \frac{p}{2},\\
=&0\text{ otherwise }.
\end{align*}
It then follows that $d_\square(W_{pA_p},W)\to  W$, and $d_{L_1}(w_{\bf D}, \psi)\to 0$. The desired conclusion then follows from Theorem \ref{thm:var_conv_new} as before.

\item[(b)]
Since $\limsup\limits_{p\to\infty}\max_{i\in [p]}\sum_{j\ne i}|A_p(i,j)|=\frac{1}{2}$, the result is immediate from Corollary \ref{cor:eg} and Lemma \ref{lemma:pure_state}.

\end{enumerate}
\end{proof}

\section{Relevant concentration inequalities} 
%
%

\begin{lemma}
\label{lemma:concentration}
Let $B_{ik} \sim \mathrm{Ber}(G(i/n, j/p)/p)$ are independent random variables with $G(x,y) \leq \lambda$. For any $\delta>0$, there exists $C>0$ (depending on $\delta$) such that 
\begin{align}
\mathbb{P}\Big(\max_{S,T \subset [p]} \Big| \sum_{k \in S, l \in T} (A_p(k, l) - \mathbb{E}[A_p(k,l)] ) \Big|>  p\delta  \Big)  \leq 2^p \exp\Big(- C \sqrt{ n  }  \Big). \nonumber
\end{align} 
\end{lemma} 

To prove Lemma \ref{lemma:concentration}, we first obsere that if $X$ is sub-exponential, then $X^2$ is sub-Weibull \cite{kuchibhotla2018moving}. 

\begin{lemma}
\label{lemma:orlicz} 
Let $Y_1, \cdots, Y_n$ be independent sub-exponential random variables with common sub-exponential parameter $c>0$.  For any $\delta>0$, there exists $C:=C(\delta)>0$ such that 
\begin{align}
\mathbb{P}\Big[ \Big| \sum_{i=1}^{n}  (Y_i^2 - \mathbb{E}[Y_i^2]) \Big| > n \delta \Big] \leq \exp\Big(- C \sqrt{n} \Big). \nonumber 
\end{align} 
\end{lemma} 

\begin{proof}
We first claim that $Y_i^2$ are sub-Weibull, i.e., there exists a constant $C'>0$ (depending only on $c$) such that 
\begin{align}
\mathbb{P}[| Y_i^2 - \mathbb{E}[Y_i^2] | >   t] \leq 2 \exp(- C'\sqrt{t}). \label{eq:sub_weibull} 
\end{align}  
We establish the upper tail deviation bound. A similar argument works for the lower tail, and is thus omitted. As the variables $Y_i$ have a common sub-exponential constant, $\max_{i\leq n} \mathbb{E}[Y_i^2]$ is uniformly bounded in $n$. 
For any $t>0$ with $\sqrt{t}  \geq (\sqrt{2}-1)^2 \max_{i \leq n} \mathbb{E}^2[Y_i]$, 
\begin{align}
\mathbb{P}[Y_i^2 > \mathbb{E}[Y_i^2] + t] &= \mathbb{P}[Y_i - \mathbb{E}[Y_i]  > \sqrt{t + \mathbb{E}[Y_i^2]} - \mathbb{E}[Y_i]] \leq \mathbb{P}[Y_i - \mathbb{E}[Y_i]  > \sqrt{t + \mathbb{E}^2[Y_i]} - \mathbb{E}[Y_i]]. \nonumber
\end{align} 
We now claim that 
\begin{align}
\sqrt{t + \mathbb{E}^2[Y_i]} - \mathbb{E}[Y_i] > \frac{\sqrt{t} - (\sqrt{2} -1) \mathbb{E}[Y_i]}{\sqrt{2} }. \label{eq:conc_int1} 
\end{align} 
Using \eqref{eq:conc_int1}, along with the deviation bound above, we have, 
\begin{align}
\mathbb{P}[Y_i^2 > \mathbb{E}[Y_i^2] + t] \leq \mathbb{P}\Big[Y_i - \mathbb{E}[Y_i]  > \frac{\sqrt{t} - (\sqrt{2}-1) \mathbb{E}[Y_i]}{\sqrt{2}}  \Big] \leq \exp(-C' \sqrt{t}), \nonumber 
\end{align}
for some constant $C'>0$. The last inequality uses that $Y_i$ is sub-exponential.  \eqref{eq:conc_int1} can be verified by direct computation. 

Using \eqref{eq:sub_weibull} and \cite[Proposition A.3]{kuchibhotla2018moving}, we have, $ \max_{i \leq n} \|Y_i^2\|_{\psi_{1/2}}$ is uniformly bounded in $n$. Consequently, using \cite[Theorem 3.1]{kuchibhotla2018moving}, for any $t>0$, we have, 
\begin{align}
\mathbb{P}\Big[ \Big| \sum_{i=1}^{n} (Y_i^2 - \mathbb{E}[Y_i^2]) \Big|  > C_1 \|b \|_2 ( \sqrt{t} + L_n t^2)  \Big] \leq \exp(- t),  \nonumber
\end{align} 
where $C_1$ is a constant free of $n$, and $L_n = C_2 \|b\|_{\infty}/\| b\|_2$ for some constant $C_2>0$ independent of $n$. Here, $b = (\|Y_1^2\|_{\psi_{1/2}} ,\cdots, \|Y_n^2\|_{\psi_{1/2}})$. We set $C_1 \|b\|_2 (\sqrt{t} + L_n t^2) = n \delta$ for some $\varepsilon >0$. Direct calculation yields that $t(\delta) = \sqrt{\frac{n\delta}{C_1C_2 \|b\|_{\infty} }} (1+ o(1))$, so that 
\begin{align}
\mathbb{P}\Big[ \Big| \sum_{i=1}^{n} (Y_i^2 - \mathbb{E}[Y_i^2]) \Big| > n \delta \Big]
 \leq \exp\Big(- C_3\sqrt{n} \Big), \nonumber
\end{align} 
where $C_3  >0$ depends on $\delta$. This concludes the proof. 
\end{proof}

\begin{proof}[Proof of Lemma \ref{lemma:concentration}]
 Recall that $B_{ik} \sim \mathrm{Ber}(G(i/n, j/p)/p)$ are independent random variables. For $S \subseteq [p]$, define $Y_i(S) = \sum_{k \in S} B_{ik}$. By Chernoff bound, the collection $\{ Y_i (S): 1\leq i \leq n, S \subseteq p\}$ are sub-exponential with a common sub-exponential parameter (depending only on $\lambda$).

 By Chernoff bounds, for any fixed $S \subseteq [p]$, $Y_i= \sum_{k \in S} B_{ik}$ is a sub-exponential random variable. Using Lemma \ref{lemma:orlicz}, 
 \begin{align}
\mathbb{P}\Big( \Big| \sum_{k, l \in S} (A_p(k,l) - \mathbb{E}[A_p(k,l)]) \Big|> p\delta  \Big) = \mathbb{P}\Big( \Big|\sum_{i=1}^{n} (Y_i(S)^2 - \mathbb{E}[Y_i(S)^2])  \Big| > n \delta  \Big) \leq \exp\Big(- C \sqrt{ n  }  \Big). \nonumber  
\end{align} 
A union bound then concludes 
 \begin{align}
\mathbb{P}\Big( \max_{S \subseteq [p]} \Big| \sum_{k, l \in S} (A_p(k,l) - \mathbb{E}[A_p(k,l)] ) \Big|> p\delta  \Big) \leq 2^p \exp\Big(- C \sqrt{ n } \Big). \label{eq:deviation1}  
\end{align}
We set $\varepsilon_{kl} := A_p(k,l) - \mathbb{E}[A_p(k,l)]$ and claim that 
\begin{align}
\max_{S,T \subseteq [p]}  |\sum_{k \in S, l \in T} \varepsilon_{kl} |  \leq \frac{5}{2}   \max_{S \subseteq [p]}  |\sum_{k,l  \in S} \varepsilon_{kl} |  \label{eq:inclusion1}  
\end{align} 
 The required conclusion follows from \eqref{eq:inclusion1} and the deviation bound \eqref{eq:deviation1}. It thus remains to prove \eqref{eq:inclusion1}. To this effect, note that 
\begin{align}
\sum_{k \in S, l \in T} \varepsilon_{kl}  &= \sum_{k \in S \backslash T, l \in T \backslash S} \varepsilon_{kl} + \sum_{k \in S \backslash T, l \in S \cap T} \varepsilon_{kl} + \sum_{k \in S \cap T, l \in  T \backslash S}\varepsilon_{kl} + \sum_{k \in S \cap T, l \in  S \cap T }\varepsilon_{kl}. \nonumber \\
\sum_{k,l \in S \cup T } \varepsilon_{kl}  &= \sum_{k,l \in S \backslash T} \varepsilon_{kl}  +  \sum_{k,l \in T \backslash S } \varepsilon_{kl}  + \sum_{k, l \in S \cap T  } \varepsilon_{kl} + 2 \Big(  \sum_{k \in S \backslash T, l \in T \backslash S} \varepsilon_{kl} + \sum_{k \in S \backslash T, l \in S \cap T} \varepsilon_{kl} + \sum_{k \in S \cap T, l \in  T \backslash S}\varepsilon_{kl}\Big). \nonumber 
\end{align} 
Thus 
\begin{align} 
\sum_{k \in S, l \in T} \varepsilon_{kl} = \frac{1}{2} \Big[  \sum_{k,l \in S \cup T } \varepsilon_{kl} - \Big( \sum_{k,l \in S \backslash T} \varepsilon_{kl}  +  \sum_{k,l \in T \backslash S } \varepsilon_{kl}  + \sum_{k, l \in S \cap T  } \varepsilon_{kl} \Big)  \Big] + \sum_{k \in S \cap T, l \in  S \cap T }\varepsilon_{kl} \nonumber 
\end{align} 
which on using triangle inequality gives \eqref{eq:inclusion1}. 

\end{proof}

\end{document}